\documentclass[11pt, oneside]{article}   	% use "amsart" instead of "article" for AMSLaTeX format
\usepackage{geometry}                		% See geometry.pdf to learn the layout options. There are lots.
\usepackage{graphicx}				% Use pdf, png, jpg, or eps§ with pdflatex; use eps in DVI mode
\usepackage{amssymb,amsmath}
\usepackage{color}
\usepackage{dsfont}

\usepackage[linecolor=white,backgroundcolor=white,bordercolor=white]{todonotes}

\usepackage[T1]{fontenc}
\usepackage[applemac]{inputenc}	
\usepackage{amsthm}
\usepackage{mathrsfs} 
\usepackage{caption}
\usepackage{subcaption}
\usepackage[colorlinks=true,linkcolor=cyan,citecolor=magenta,bookmarks=true]{hyperref}

\newtheorem{thm}{Theorem}[section]
\newtheorem{prop}[thm]{Proposition}
\newtheorem{coro}[thm]{Corollary}
\newtheorem{lem}[thm]{Lemma}
\newtheorem{conj}[thm]{Conjecture}

\newtheorem{ex}[thm]{Example}

       \theoremstyle{definition}
       \newtheorem{dfn}{Definition}[section]
        
       \theoremstyle{remark}
       \newtheorem*{rmk}{Remark}

\newcommand{\PP}{\mathbb{P}}
\newcommand{\E}{\mathbb{E}}
\newcommand{\R}{\mathbb{R}}
\newcommand{\C}{\mathbb{C}}
\renewcommand{\H}{\mathbb{H}}
\newcommand{\N}{\mathbb{N}}
\newcommand{\Z}{\mathbb{Z}}
\newcommand{\U}{\mathrm{U}}
\newcommand{\SU}{\mathrm{SU}}
\newcommand{\SO}{\mathrm{SO}}
\newcommand{\Sp}{\mathrm{Sp}}
\newcommand{\GL}{\mathrm{GL}}

\newcommand{\Sbb}{\mathbb{S}}

\newcommand{\Vbb}{\mathbb{V}}
\newcommand{\Ebb}{\mathbb{E}}
\newcommand{\Fbb}{\mathbb{F}}

\newcommand{\Ld}{\mathrm{L}}

\newcommand{\sym}{\mathrm{sym}}
\newcommand{\YM}{\mathrm{YM}}
\newcommand{\Tr}{\mathrm{Tr}}
\newcommand{\tr}{\mathrm{tr}}

\newcommand{\build}[3]{\mathrel{\mathop{\kern 0pt#1}\limits_{#2}^{#3}}}

\def\br{\begin{color}{blue}}
\def\bb{\begin{color}{blue}}
\def\bg{\begin{color}{blue}}
\def\er{\end{color}}
\def\eg{\end{color}}
\def\eb{\end{color}}
\def\a{\alpha}

\def\b{\beta}

\def\ve{\varepsilon}

\def\g{\gamma}

\def\l{\lambda}
\def\L{\Lambda}
\def\P{\mathbb{P}}

\def\mfk{\mathfrak}

\def\cG{\mathcal{G}}

\def\<{\langle}
\def\>{\rangle}
\def\End{\mathrm{End}}

\def\pl{\partial}
\def\lto{\longrightarrow}

\def\bth{\begin{thm}}
\def\eth{\end{thm}}

\author{Antoine Dahlqvist\thanks{University of Sussex, School of Mathematical and Physical Sciences, Pevensey 3 Building, Brighton, UK} \and Thibaut Lemoine\thanks{Universit\'e de Lille, CNRS, UMR 9189 - CRIStAL, 59651 Villeneuve d'Ascq, France}}

\title{Large $N$ limit of Yang--Mills partition function and\\ Wilson loops on compact surfaces}
\date{}							% Activate to display a given date or no date

\begin{document}
\bibliographystyle{plain}
\maketitle
\abstract{We compute the large $N$ limit of several objects related to the two-dimensional Euclidean Yang--Mills measure on  closed, connected, orientable surfaces $\Sigma $ with genus $g\geq 1$, when a structure group is taken among 
the classical groups of order $N$. We first generalise to all classical groups the convergence of partitions functions obtained by the second author for unitary groups. We then apply this result to prove convergence of Wilson loop 
observables for loops included within a topological disc of $\Sigma$.  This convergence solves a conjecture of B. Hall  and shows moreover that the limit is independent of the topology of $\Sigma$ and is equal to an evaluation of the 
planar 
master field. Besides, using similar arguments, we show that Wilson loops vanish asymptotically for all non-contractible simple loops.}

%\section{}
%\subsection{}

\section{Introduction}

The purpose of this article is to study a model of random matrices originating from 2D Euclidean quantum field theory, known as the two-dimensional Yang--Mills measure. This model is defined for a classical compact Lie group and a 
compact surface endowed with an area measure and can be understood as a gauge theory. Each loop on the surface 
defines an observable called 
Wilson loop. We are interested here in Wilson 
loops when 
the size  $N$ of the matrices goes to infinity, while the loop and the surface are fixed.

Let us attempt to give an account of the origin and state of this problem. The question seems to have first appeared in a mathematical paper in \cite{Sin}. Therein, the candidate limit of Wilson loops on an arbitrary surface $\Sigma$ 
is called 
master 
field. This paper was inspired by the large $N$ limits considered in quantum gauge theories \cite{DK2,KK,Mig,MM,GrossMat,Gop}, which started after the landmark work of t'Hooft \cite{Hoo}. Since then, the Yang--Mills measure in two 
dimensions 
has been rigorously defined\footnote{See \cite{LS} for a historical review on the motivations of these approaches.} \cite{Gro,Dri,GKS,Sen4,Sen0,Sen,Lev3}, and the latter question has received the attention of several 
mathematicians\footnote{See \cite{Lev7} for a recent reviews paper on the Yang--Mills 
measure and the master fields in two dimensions. See also \cite{Chatt} for a study of large $N$ limits of Wilson loops, for discrete Yang--Mills measure in dimension higher than $2$.}   \cite{Xu,Sen2,AS,Lev,DHK,DGHK,DN,Hal2}. It
is now
known that the limit indeed exists when the surface is the plane \cite{AS,Lev2} or the sphere \cite{DN}, giving rise to two different objects. 
 Simultaneously to and independently from \cite{DN}, conditional convergence results were obtained in \cite{Hal2}  for loops within a topological disc embedded in a compact surface, with different set of assumptions for simple 
 loops. Both works \cite{DN,Hal2} relied on the so-called Makeenko--Migdal equations, which allow to tackle the problem recursively in the complexity of loops. These equations first appeared in physics in \cite{MM,KK} and were proved 
 rigorously on the plane in \cite{Lev2,Dah2,DHK} and on surfaces in \cite{DHK,DGHK}.

Our main result, Theorem \ref{__THM: Disc YMCont}, shows that for all loops included within a topological disc of $\Sigma$, the limit Wilson loop converges and its limit is given by the value of the \emph{master field on the plane} evaluated at the loop obtained by 
embedding the same disc  in the plane instead  of $\Sigma$, in a way that preserves the area measure. In that  sense, the behaviour of the Yang--Mills measure within a topological disc is not affected asymptotically by the topological constraint imposed by the 
surface.  The idea of a "trivial" behaviour of the large N limit of the Yang--Mills measure on closed surfaces of genus $g\ge 1$  appeared previously in the physics literature.\footnote{at least for the partition function, see \cite[Section 3]{Dou} or \cite{Rus}.} 
Nonetheless, we could not find mentioned the specific manifestation of this triviality appearing in our main result. 

Our second result Theorem \ref{__THM:simple non-contrac} shows that for any simple non-contractible loop the Wilson loop of any iteration vanishes asymptotically. This leaves the case of loops which are neither embedded in a 
topological disc nor an iteration of a simple loop. We investigate this question in the next paper in the series \cite{DL}. 

Our argument does not rely here on the Makeenko--Migdal equations but mainly on the convergence of another quantity: the partition function. We prove the convergence of the partition function using harmonic analysis on the classical 
compact groups, generalising a result of \cite{Gur} and of the second author \cite{Lem}.

% The purpose of this paper is double: our initial goal, which was to investigate large $N$ limit of several observables in Yang--Mills theory, led us to consider more general probabilistic models, which are respectively Gaussian random variables in a lattice and Brownian motions on groups pinned by word measures. We will first motivate the use of these 
%models by explaining a few generalities about Euclidean Yang--Mills theory on compact surfaces, and state the main results we will prove in this setting, then we will describe the tools used to achieve these proofs.

The rest of the paper is divided in three sections. The first one recalls the setup of the problem, presents the main results and an idea of proof. The second section recalls the necessary notions from representation theory of compact groups and the result of 
\cite{Gur,Lem} on partition functions which we generalise to all group series of compact matrix Lie groups and all area parameter. The last one gives the proof of our main results on Wilson loops. 

\tableofcontents
\section{Setup and statement of results}

\subsection{Heat kernel on compact Lie groups}

\subsubsection{Heat kernel}
 In this text, $G$ will denote a compact Lie group endowed with a bi-invariant inner product. The heat kernel on $G$ is the family of smooth functions $(p_t,t\in (0,\infty))$ on $G$ satisfying 
$$\frac d {dt} p_t(g)= \frac{1}{2} \Delta (p_t)(g) \text{ for } t>0 \text{ and }g\in G $$
and for any continuous function $f,$
$$\lim_{t\to 0}\int_G  f(g)p_t(g)dg= f(1).$$
We denote here by $dg$ and  $\Delta$  the Haar probability measure on $G$ and   the Laplace--Beltrami operator associated to its inner product, while we write $1$ for the identity element of $G.$ The functions $p_t$ are central and invariant by inversion for all 
$t>0,$ and form a semigroup for 
the convolution product on $G$: that is, for $t>0,$
$$ p_{t}(xgx^{-1})= p_t(g) \text{ and }p_t(g^{-1})=p_t(g) \hspace{0,2 cm}\forall x,g\in G,$$
and
$$p_t*p_s=p_{t+s},\forall s>0.$$

\subsubsection{Classical compact Lie groups}  
\label{sec:Intro Class Gp}
In the sequel, we shall say that $G_N$ is a \emph{classical group of size} $N$\footnote{Although $\Sp(N)$ is a group of complex matrices of size $2N$, the denomination ``size $N$'' is not misleading, as it can be also considered as a subgroup of $\GL_N(\H)$. We do not exploit this property, and choose such a terminology only for the sake of simplicity.} if it is equal to one of the following matrix Lie groups:
\begin{enumerate}
\item The unitary group $\U(N)=\{U\in\GL_N(\C):UU^*=I_N\}$,
\item The special unitary group $\SU(N)=\{U\in\U(N):\det(U)=1\}$,
\item The special orthogonal group $\SO(N)=\{O\in\GL_N(\R):OO^t=I_N\}$,
\item The compact symplectic group $\Sp(N)=U(2N)\cap\{S\in\GL_{2N}(\C):S^tJS=J\}$ where $J\in\GL_{2N}(\C)$ is defined by
\[
J=\begin{pmatrix}
0 & I_N\\
-I_N & 0
\end{pmatrix}.
\]
\end{enumerate}

We fix the bi-invariant metric as follows. Assume that $G_N$ is a subgroup of $\GL_n(\C)$ (so that $n=N$ for $\U(N)$, $\SU(N)$, $\SO(N)$ and $n=2N$ for $\Sp(N)$). We set an integer parameter\footnote{This parameter corresponds to the Dyson index in random matrix theory. For details over its significance, see for instance Section 4.1 of \cite{AGZ}. We introduce this parameter so that standard Brownian motions on these Lie algebras all converge to the same process when $N\to\infty$, see for instance \cite[Section 2]{Dah} for an explanation.} $\beta$ which is equal to 1 if $G_N=\SO(N)$, 2 if $G_N=\U(N)$, $\SU(N)$, and 4 if $G_N=\Sp(N)$. We endow the Lie algebra $\mathfrak{g}_N$ of $G_N$ with the scalar product

\begin{equation}
\langle X,Y\rangle = \frac{\beta n}{2}\Tr(X^*Y),\forall X,Y\in \mfk g_N.\label{------eq:inner P LieAlg}
\end{equation}
Here $\Tr$ denotes the trace of the above matrix, that is the non-normalised sum of diagonal coefficients. It is conjugation-invariant in the sense that
$$\langle g Xg^{-1},gYg^{-1}\rangle= \<X,Y\>, \text{ for all } X,Y\in \mfk g_N \text{ and }g\in G_N,$$
and defines therefore a bi-invariant metric on $G_N.$ 

\begin{rmk} Except in case 1., there is up to constant a unique invariant inner product. The above choice of scaling is standard in random matrices. For instance, the gaussian vector on Hermitian matrices obtained by composition of 
the above scalar product with the multiplication by $i$ is the classical Gaussian Unitary Ensemble.
\end{rmk}

\subsection{2D-maps and multiplicative functions}

Assume that $\Sigma$ is a two dimensional compact Riemannian manifold, with Riemannian distance $d_\Sigma$ and genus $g\geq 0$. It is homeomorphic to a $4g$-gon whose sides are, counterclockwise, $a_1,b_1,a_1^{-1},b_1^{-1},\ldots,a_g,b_g,a_g^{-1},b_g^{-1}$, such that each two edges with same letter are glued together, while respecting the orientation. The polygon is called the \emph{fundamental domain} of the surface. If $\Sigma$ has a nonempty boundary $\partial \Sigma$, the latter can be described as the union of $n$ connected components $(C_i)_{1\leq i\leq n}$, each homeomorphic to the unit circle $S^1$. We introduce here some notations used throughout this text. 

Denote by $\mathrm{P}(\Sigma)$ the set of continuous maps $[0,1]\to \Sigma$ with positive and finite length, up to Lipschitz reparametrisation. When $\g\in \mathrm{P}(\Sigma) $, $\mathscr{L}(\g)$ denotes its length,   $\g^{-1}\in \mathrm{P}(\Sigma)$ its reverse, $\underline \g= \tilde \g(0)$ its starting point and $\overline \g=\tilde \g (1)$ its endpoint, where $\tilde \g$ is some parametrisation. When $\overline \g=\underline \g$, we say that $\g$ is a \emph{loop} of $\Sigma$ and write $$\mathrm{L}(\Sigma)=\{\g\in \mathrm{P}(\Sigma): \underline{\g}=\overline{\g}\}.$$ 
When $\g_1,\g_2\in \mathrm{P}(\Sigma)$ with $\underline{\g}_2=\overline{\g}_1,$ $\g_1\g_2\in \mathrm{P}(\Sigma)$ stands for their concatenation. A distance on $\mathrm{P}(\Sigma) $ is defined \cite{Lev2} setting for all $\g_1,\g_2\in \mathrm{P}(\Sigma),$
\begin{equation}
d(\g_1,\g_2)= |\mathscr{L}(\g_1)-\mathscr{L}(\g_2)| +\inf \inf_{t\in [0,1]}\{ d_{\Sigma}(\tilde\g_1(t),\tilde \g_2(t)) \}\label{Distance Loops}
\end{equation}
where the first infimum is taken over all Lipschitz parametrisations $\tilde \g_1,\tilde \g_2$ of respectively $\g_1$ and $\g_2$. In the following paragraphs, we will consider subsets of $\mathrm{P}(\Sigma)$ corresponding to loops traced in embedded graphs.

%If we endow $\Sigma$ with a finite set $\mathcal{S}$ of paths of $\build{\Sigma}{}{\circ}$ that are either disjoint or equal up to orientation, such that $\gamma\in\mathcal{S}$ if and only if $\gamma^{-1}\in\mathcal{S}$, then the couple $(\Sigma,\mathcal{S})$ is called a \emph{marked surface}. Denote by $\mathcal{B}(\Sigma)$ the set of connected components of $\partial\Sigma$, each taken twice, one for each orientation. There is a natural action of $\Z/2\Z$ on $\mathcal{S}\cup\mathcal{B}(\Sigma)$ by orientation-reversal.
 
 If $\Psi$ is a diffeomorphism of $\Sigma$ and $\ell\in \Ld(\Sigma),$   $\psi(\ell)\in \Ld(\Sigma)$ is the loop of $\Sigma$ obtained by composition of a parametrisation of $\ell$ with $\Psi$.

\subsubsection{Topological maps on compact surfaces}
 
We follow here conventions of \cite[Section 1.3.2]{LZ}. A \emph{graph} $\mathcal{G}$ is a triple $(\Vbb,\Ebb,I)$ consisting of a set $\Vbb$ of \emph{vertices}, a set $\Ebb$ of \emph{edges} and an \emph{incidence relation} $I$ such that an edge is incident to either one vertex or two distinct vertices, called \emph{endpoints}. It can be given an \emph{orientation} by setting, for any $e\in\E$, a source $\underline{e}\in\Vbb$ and a target $\overline{e}\in\Vbb$. An oriented graph can be then represented by a quadruple $(\Vbb,\Ebb,s,t)$ where $\Vbb$ is the set of vertices, $\Ebb$ the set of oriented edges, and $s,t:\Ebb\to\Vbb$ are the functions that map respectively an edge $e$ to its source and target. An \emph{isomorphism} between two graphs $\mathcal{G}_1=(\Vbb_1,\Ebb_1,s_1,t_1)$ and $\mathcal{G}_2=(\Vbb_2,\Ebb_2,s_2,t_2)$ is a bijection $\phi:\Vbb_1\cup\Ebb_1\to\Vbb_2\cup\Ebb_2$ that sends $\Vbb_1$ (resp. $\Ebb_1$) onto $\Vbb_2$ (resp. $\Ebb_2$), and such that
\[
s_2(\phi(e))=\phi(s_1(e)), \ t_2(\phi(e))=\phi(t_1(e)),\ \forall e\in\Ebb.
\]

A \emph{topological map} $M$ on $\Sigma$ is a finite oriented graph $\mathcal{G}=(\Vbb,\Ebb,s,t)$ endowed with an embedding $\theta:\mathcal{G}\to\Sigma$, called \emph{drawing of the graph}, such that:
\begin{itemize}
\item $\Sigma$ is a compact, connected and orientable surface,
\item the vertices are drawn as distinct points of $\Sigma$,
\item oriented edges are drawn as oriented continuous curves that only intersect at their endpoints,
\item for any $e\in\Ebb$, there is an edge $e^{-1}\in\Ebb$ such that $\theta(e^{-1})=\theta(e)^{-1}$,
\item the set $\Fbb=\Sigma\setminus\bigcup_{e\in\Ebb} \{\theta(e)\}$ is given by a union of open discs called \emph{faces} of the map.
\end{itemize}

In order to avoid too heavy notations, we will always identify vertices and edges with their drawing, so that a map on $\Sigma$ can be represented by $(\Vbb,\Ebb,\Fbb)$ (the applications $s$ and $t$ will remain implicit from now on). $\Sigma$ is called the \emph{underlying surface} of $M$.  If $v\in\Vbb$ is a vertex in $M$, we denote by $\mathrm{P}(M)$ (resp. $\mathrm{L}(M)$, $\mathrm{L}_v(M)$) the set of paths (resp. loops, loops with base $v$) in $M$ obtained by concatenation of oriented edges. A loop is simple when each vertex of $M$ is the source of at most one of its edges.  

\begin{figure}[!h]
\centering
\includegraphics[scale=0.8]{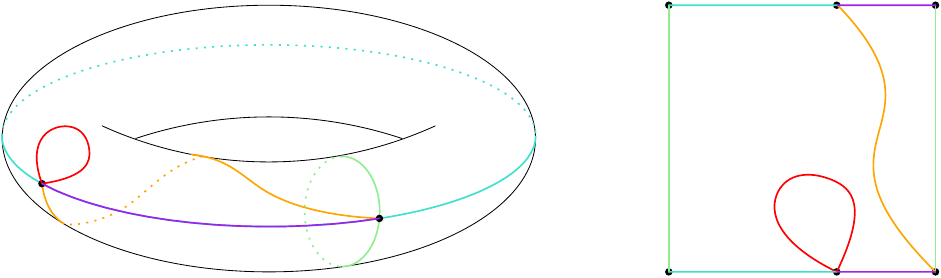}
\caption{\small An example of topological map on a torus (on the left), and its representation as an abstract graph (on the right) whose opposite edges of the same colour are identified.}\label{fig:maps}
\end{figure}

Two maps $M_1=(\Vbb_1,\Ebb_1,\Fbb_1)$ and $M_2=(\Vbb_2,\Ebb_2,\Fbb_2)$, respectively on $\Sigma_1$ and $\Sigma_2$ and with associated graphs $\mathcal{G}_1 $ and $\mathcal{G}_2$, are \emph{equivalent} if there is an orientation-preserving homeomorphism $\phi:\Sigma_1\to\Sigma_2$ such that the restriction of $\phi$ to $M_1$ induces an isomorphism between $\mathcal{G}_1$ and $\mathcal{G}_2$. The \emph{genus} of a topological map in $\Sigma$ is the genus of $\Sigma$, and it does not depend on its equivalence class.

% Note that Dehn twists, that generate the mapping class group of $\Sigma$ (see \cite{Hum}), preserve the equivalence class of a map on $\Sigma$.

Mind that in the above definition, the boundary of $\Sigma$ might be non-empty. The following result, which is a particular case of Proposition 1.3.10 in \cite{Lev2}, allows to  identify boundary components of $\Sigma$ with loops of a map, up to rerooting. 

\begin{prop}\label{Prop: Boundary}
Assume that $M$ is a topological map on a compact, connected, orientable surface $\Sigma$ endowed with a Riemannian metric,  such that $\partial\Sigma$ has positive and finite length. Then, any connected component $C$ of $\partial\Sigma$ is the drawing of an element of $\mathrm{L}(M)$.
\end{prop}
When there is a metric on $\Sigma$ such that each edge of  a topological map $M$ is mapped to a curve of finite length,  we call $M$  \emph{rectifiable.} When the boundary of $\Sigma$ is empty, we shall say that the map $M$ is \emph{closed}. This latter 
property does not depend on the equivalence class of $M.$

\subsubsection{Topological maps on $\R^2$ or an open disc}
A topological map $M$ on the \emph{plane $\R^2$ or an open disc} is a topological map on the sphere $\Sbb^2$ together with a marked face $f_\infty$. The latter is called the infinite face of $M$, and it contains the point at infinity when one makes the identification $\Sbb^2\simeq\R^2\cup\{\infty\}$.  Unless mentioned otherwise, when $M$ is a topological map on the plane or an open disc, $\Fbb$ will denote the set  of \emph{bounded} faces, and $\hat{\Fbb}$ will denote $\Fbb\cup\{f_\infty\}$ whenever it makes sense.

%From now on, for any map $M$, be it on the plane, a disc or a closed, connected, orientable surface, $\Fbb$ will denote by default the set of \emph{bounded} faces, and $\hat{\Fbb}$ will denote $\Fbb\cup\{f_\infty\}$ whenever it makes sense.

\subsubsection{Multiplicative functions}

For any $P\subset\mathrm{P}(\Sigma)$ and any compact group $G$, a \emph{multiplicative function} $h:P\to G$ is a function that satisfies $h_{\gamma^{-1}}={h_\gamma}^{-1}$ for any $\gamma\in P$ such that $\gamma^{-1}\in P$, and $h_{\gamma_1\gamma_2}=h_{\gamma_2}h_{\gamma_1}$ for any $\gamma_1,\gamma_2\in P$ such that $\gamma_1\gamma_2\in P$. We denote by $\mathcal{M}(P,G)$ the space of multiplicative functions from $P$ to $G$. The gauge group $\Gamma=\mathscr{C}^\infty(\Sigma,G)$ acts on it by \emph{gauge transformations}: for any $j\in\Gamma$ and $h\in\mathcal{M}(P,G)$,
\[
(j\cdot h)_\gamma = j(\overline{\gamma})^{-1} h_\gamma j(\underline{\gamma}),\ \forall \gamma\in P.
\]
Considering topological maps, gauge-invariance of multiplicative functions is equivalent to a $G$-invariance, as stated by the following lemma which is a particular instance of \cite[Lemma 2.1.5]{Lev2}.

\begin{lem}
Let $M$ be a topological map in $\Sigma$,  $v\in\Vbb$ and $\gamma_1,\ldots,\gamma_n\in\mathrm{P}(M)$. There exist $\ell_1,\ldots,\ell_m\in\mathrm{L}_v(M)$ such that for any $f:G^n\to\C$ with  $h\mapsto f(h_{\gamma_1},\ldots,h_{\gamma_n})$  gauge-invariant on $\mathcal{M}(\mathrm{P}(M),G)$, 
\[
f(h_{\gamma_1},\ldots,h_{\gamma_n})=\tilde{f}(h_{\ell_1},\ldots,h_{\ell_m}),\ \forall h\in\mathcal{M}(\mathrm{P}(M),G),
\]
for some function $\tilde{f}:G^m\to\C$ which is invariant by the diagonal action of $G$.
\end{lem}

Two sigma-fields may be put on $\mathcal{M}(\mathrm{P}(M),G)$:
\begin{itemize}
\item The smallest sigma-field $\mathcal{C}$ such that for any $\gamma\in\mathrm{P}(M)$ the evaluation function\footnote{It is also called \emph{holonomy} in the literature.}
\[
H_\gamma:\left\lbrace\begin{array}{ccc}
\mathcal{M}(\mathrm{P}(M),G) & \to & G\\
h & \mapsto & h_\gamma
\end{array}\right.
\]
is measurable: it is the \emph{cylindrical sigma-field}.
\item The smallest sigma-field $\mathcal{J}$ that makes 
\[
h\mapsto f(h_{\ell_1},\ldots,h_{\ell_n})
\]
measurable, for all $v\in\Vbb$, $n\in\N$, $\ell_1,\ldots,\ell_n\in\mathrm{L}_v(M)$ and $f:G^n\to\C$ $G$-invariant: it is the \emph{invariant sigma-field}.
\end{itemize}
We will mainly work with $\mathcal{C}$ but some results will only hold on $\mathcal{J}$, therefore we will specify which sigma-field we consider.

\subsubsection{Wilson loops}

If $\chi: G\to \C$ is a central function, for any $\ell\in \Ld(M)$, $\chi(h_{\ell})$ does not depend on the choice of based point. The function 
\[
\begin{array}{ccl}&\mathcal{M}(\mathrm{L}(M),G) &\lto \C \\ &h&\longmapsto \chi (h_\ell)\end{array}
\]
is then called a Wilson loop. 
When $G$ belongs to one of the four series given above, we will be interested in the Wilson loops 
\begin{equation}
W_\ell:\left\lbrace\begin{array}{ccl} \mathcal{M}(\mathrm{L}(M),G) & \lto & \C \\ h&\longmapsto & \tr (h_\ell),\end{array} \right. \label{------eq: Wilson loop}
\end{equation}
where $\ell$ is a loop of $M$, while for any $M\in  M_d(\C)$, $\tr(M)=\frac{1}{d}\sum_{i=1}^d M_{i,i}.$

\begin{rmk}[Gauge equivalence] For most compact Lie groups $G$, it can be shown that the family of Wilson loops separate points of $\mathcal{M}(\mathrm{P}(M),G)/\Gamma$ endowed with the quotient topology.  When $G$ belongs to one of the four series of the previous section, it can further be shown \cite{Sen5,Lev4} that it is enough to consider the family $\{W_\ell, \ell \in \Ld(M)\}.$ 
\end{rmk}

\subsubsection{Area-weighted maps}

\emph{An area-weighted map} is a topological map $M$ together with\footnote{In  the case of $\R^2$, the marked faces are excluded because their area is  considered to be  infinite. } a function $a: \Fbb \to \R_+^{*}.$ Two maps $(M,a)$ and $(M',a')$ are \emph{equivalent} if $M$ and $M'$ are equivalent as maps and the associated 
homomorphism $\phi$ of $\Sigma$ defines a bijection $\mfk F: \Fbb\to \Fbb'$ with 
$$a' \circ\mfk F= a. $$ 
For any $T>0$ we set 
$$\Delta_M(T)=\{a: \Fbb\to \R_+^{*}: \sum_{f\in \Fbb} a_f=T\}$$
when $M$ is a closed topological map, and when $(M,a)$ is an area-weighted map we write $|a|=\sum_{f\in \Fbb} a_f.$ If the surface is endowed with a Riemannian volume $\mathrm{vol}$, then it induces in particular an area function 
$\mathrm{vol}:\Fbb\to\R_+^{*}$.

\subsection{Two dimensional Yang--Mills measure}\label{sec:YM Intro}

In this section we recall a definition of the  Yang--Mills measure in three steps:
\begin{enumerate}
\item Given a topological map $M$, we define a uniform measure $\mathrm{U}_{M,C,G}$ on $\mathcal{M}(\mathrm{P}(M),G)$ with and without constraints.
\item We define the discrete Yang--Mills measure $\YM_{M,C,a,G}$ on an area-weighted topological map $(M,a)$ as an absolutely continuous measure with respect to $\mathrm{U}_{M,C,G}$.
\item We define the Yang--Mills holonomy field $(H_\ell)_{\ell\in\mathrm{P}(\Sigma)}$ on any compact, connected, orientable Riemann surface $\Sigma$ with volume form $\mathrm{vol}$, whose distribution $\YM_{\Sigma,C,G}$ is the continuous version of $\YM_{M,C,a,G}.$
\end{enumerate}

\subsubsection{Uniform measure on $\mathcal{M}(\mathrm{P}(M),G)$}

For any topological map $M$, the space $\mathcal{M}(\mathrm{P}(M),G)$ is a compact Lie group when endowed with the pointwise multiplication and has a unique Haar measure $\mu_M,$ on both sigma-fields $\mathcal{C}$ and $\mathcal{J}$.  Fixing a set $\Ebb^+$ of  positively oriented edges (so that for any $e\in\Ebb$, only $e$ or $e^{-1}$ belongs to $\Ebb^+$), $\mathcal{M}(\mathrm{P}(M),G)$ can  be identified with $G^{\Ebb^+} $ and $\mu_M$ is the push-forward of the direct product of Haar measures on $G$. Let us call it the \emph{unconstrained uniform measure} and denote it by $\mathrm{U}_{M,G}$. In contrast, we need to define a constrained version to put boundary conditions when necessary. If $\Sigma$ is a compact surface with boundary, denote by $\mathcal{B}(\Sigma)$ the set of connected components of $\partial\Sigma$, each taken twice, one for each orientation. There is a natural action of $\Z/2\Z$ on $\mathcal{B}(\Sigma)$ by orientation-reversal.

A \emph{set of boundary conditions} on a compact surface $\Sigma$ is a $\Z/2\Z$-equivariant map
\[
C:\mathcal{B}(\Sigma)\to G/\mathrm{Ad},
\]
where $G/\mathrm{Ad}$ is the set of conjugation classes of $G$ where $\Z/2\Z$ acts by inversion. The fact that boundary conditions take values in $G/\mathrm{Ad}$ rather than simply $G$ will be explained later. If $(M,a)$ is an area-weighted map on $\Sigma$, then the boundary conditions also apply to $M$, as according to Proposition \ref{Prop: Boundary},  elements of $\mathcal{B}(\Sigma)$ can be identified with loops $\Ld(M)$ up to re-rooting. Denoting by $L_1,\ldots,L_p$ the elements of $\mathcal{B}(\Sigma)$ oriented positively,  there is a bijection between $(G/\mathrm{Ad})^{\mathcal{B}(\Sigma)}/(\Z/2\Z)$ and $(G/\mathrm{Ad})^p$, so that any set $C$ of boundary conditions can be identified with a tuple $C=(t_1,\ldots,t_p)$ with $t_1,\ldots,t_p\in G/\mathrm{Ad}$.

For any $t\in G/\mathrm{Ad}$ and $n\geq 1$, the set $t(n)=\{(x_1,\ldots,x_n)\in G^n:x_1\cdots x_n\in t\}$ is a homogeneous space for the $G^n$-action
\[
(g_1,\ldots,g_n)\cdot(x_1,\ldots,x_n)=(g_1x_1g_2^{-1},\ldots,g_nx_ng_1^{-1}).
\]
We denote by $\delta_{t(n)}$ the extension to $G^n$ of the unique $G^n$-invariant probability measure on $t(n)\subset G^n$. It can be thought of as the \emph{conditional Haar measure on} $G^n$, under the condition $x_1\cdots x_n\in t$.

%Here are some of its properties.
%
%\begin{prop}[\cite{Lev2}, Lemma 2.3.4]\label{prop:desintegration}
%The measures $(\delta_{\mathscr{O}(n)})_{\mathscr{O}\in\mathrm{Conj}(G),n\in\N^*}$ satisfy the following properties:
%\begin{enumerate}
%\item For any $x\in G$, if we denote by $\mathscr{O}_x\in\mathrm{Conj}(G)$ its conjugation class,
%\begin{equation}
%\delta_{\mathscr{O}_x} = \int_G\delta_{gxg^{-1}}dg.
%\end{equation}
%\item For any continuous functions $f:G\to\C$, $g:G^{n-1}\to\C$ and $h:G^{n-1}\to G$,
%\begin{align}
%\begin{split}
%\int_G fd\delta_{\mathscr{O}}&\int_{G^{n-1}} g(x_1,\ldots,x_{n-1})dx_1\cdots dx_{n-1} \\
%=&\int_{G^n} f(h(x_1,\ldots,x_{n-1})x_1\cdots x_n h(x_1,\ldots,x_{n-1})^{-1}) \\
%& g(x_1,\ldots,x_{n-1})\delta_{\mathscr{O}(n)}(dx_1\cdots dx_n). \\
%\end{split}
%\end{align}
%\item For any continuous $f:G^n\to\C$,
%\begin{equation}
%\int_{G^{n+1}} f(x_1,\ldots,x_n)\delta_{\mathscr{O}(n+1)}(dx_1\cdots dx_{n+1})=\int_{G^n} fd\delta_{\mathscr{O}(n)},
%\end{equation}
%\begin{equation}
%\int_{G^n} f(x_2,\ldots,x_n,x_1)\delta_{\mathscr{O}(n)}(dx_1\cdots dx_n) = \int_{G^n}f d\delta_{\mathscr{O}(n)},
%\end{equation}
%\begin{equation}
%\int_G\left[\int_{G^n} fd\delta_{\mathscr{O}_y(n)}\right]dy = \int_{G^n} f(x_1,\ldots,x_n)dx_1\cdots dx_n.
%\end{equation}
%\end{enumerate}
%\end{prop}

%A drawback of this conditional measure is that it is not invariant by permutation of the edges anymore.
% for instance, whereas $dg_1\otimes dg_2=dg_2\otimes dg_1$, in general  $\delta_{t(2)}(dg_1dg_2)\neq\delta_{t(2)}(dg_2dg_1)$.
Let us now define a measure on $(\mathcal{M}(\mathrm{P}(M),G),\mathcal{C}).$  Therefor, let us choose a specific labelling of $\Ebb$. If we denote again by $L_1,\ldots,L_p$ the connected components of $\partial\Sigma$, we choose the labels of the oriented edges $e\in\Ebb^+$ such that $L_i=e_{i,1}\cdots e_{i,n_i}$ for every $i$. Denote finally by $e_1,\ldots,e_m$ the remaining edges of $\Ebb^+$.  Then $\mathcal{M}(\mathrm{P}(M),G)$ is isomorphic to 
$G^{m} \times G^{n_1}\times \cdots \times G^{n_p}$ where   $h\in \mathcal{M}(\mathrm{P}(M),G)$ is mapped to the tuple $g_l=h_{e_l},1\le l\le m, g_{i,n_k}=h_{e_{i,n_k}},1\le i\le n_i, 1\le k\le p.$    The \emph{uniform measure on} $\mathcal{M}(\mathrm{P}(M),G)$ \emph{with boundary conditions} $C=(t_1,\ldots,t_p)$ is the measure $\mathrm{U}_{M,C,G}$ on $(\mathcal{M}(\mathrm{P}(M),G),\mathcal{C})$ defined by
\[
\mathrm{U}_{M,(t_1,\ldots,t_p),G}(dh)=dg_1\otimes\cdots\otimes dg_m\otimes\bigotimes_{i=1}^p\delta_{t_i(n_i)}(dg_{i,1}\cdots dg_{i,n_i}).
\]
By convention, when $C=\emptyset,$ we write $\mathrm{U}_{M,C,G}=\mu_M.$ The translation invariance of the Haar measure and the choice $G/\mathrm{Ad}$-valued boundary conditions yield the following invariance proved given in \cite[Prop. 2.3.6.]{Lev2}.

\begin{prop} The measure $\mathrm{U}_{M,(t_1,\ldots,t_p),G}$  is left invariant by the action of $\Gamma$ on $\mathcal{M}(\mathrm{P}(M),G)$. 
\end{prop}

We shall also consider constrained measures. Assume that $M$ is closed,   $\ell\in \Ld(M)$ is a simple loop and $t\in G/\mathrm{Ad}$. Denoting by $\ell= e_{0,1}\ldots e_{0,n}$ the edge decomposition of $\ell$ in $M,$  and labeling $e_1,\ldots,e_m$ the other edges of $M,$  the \emph{uniform measure on} $\mathcal{M}(\mathrm{P}(M),G)$ with \emph{constraint} $t$ on $\ell$ is 
the measure $\mathrm{U}_{M,C_{\ell\mapsto t},G}$ on $(\mathcal{M}(\mathrm{P}(M),G),\mathcal{C})$ defined by
\[
\mathrm{U}_{M,C_{\ell\mapsto t},G}(dh)=dg_1\otimes\cdots\otimes dg_m\otimes \delta_{t(n)}(dg_{0,1}\cdots dg_{0,n}).
\]

%  A gauge transformation $j\in\Gamma$ acts on the boundary conditions by
%\[
%j\cdot (t_1,\ldots,t_p)=(j(\overline{\ell})^{-1}t_1j(\underline{\ell}),\ldots,j(\overline{\ell})^{-1}t_1j(\underline{\ell})).
%\]
%Note that any $\ell\in\mathcal{B}(\Sigma)$ is a loop, thus its source and target are identical, so that the action above is an action by conjugation. It induces an action of $j\in\Gamma$ on $\mathrm{U}_{M,G,C}$ by
%\[
%j\cdot \mathrm{U}_{M,(t_1,\ldots,t_p),G}  = \mathrm{U}_{M,j\cdot(t_1,\ldots,t_p),G},
%\]
%and it appears that the measure $\mathrm{U}_{M,G,C}$ is gauge-invariant, for any topological map $M$ and set of boundary conditions $C$, by construction of the measures $\delta_{t_i(n_i)}$.

\subsubsection{Discrete Yang--Mills measure}

We follow here \S 1.5.2 of \cite{Lev3}.  Consider a compact, connected, orientable surface $\Sigma$ with genus $g\geq 0$, $p$ boundary components, and an area-weighted map $(M,a)$ with graph $\cG$ embedded in $\Sigma$. For each face $f\in\Fbb$, its boundary $\pl f$ is a loop of $\cG$ given by the concatenation of edges bordering $f$ considered up to the choice of base point and of orientation.   
%The \emph{Yang--Mills measure}\footnote{When the boundary of $\Sigma$ is non-empty, we  also call it Yang--Mills measure with free boundary conditions.} on $(M,a)$ with structure group $G$ is the probability measure $\YM_{M,a,G}$ on $(\mathcal{M}(\mathrm{P}(M),G),\mathcal{C})$ with density
%\begin{equation}
%\frac{1}{Z_{M,a,G}}\prod_{f\in \Fbb} p_{a_f}(h_{\pl f}) \label{------eq:DS form}
%\end{equation}
%with respect to $\mathrm{U}_{M,C,G}$, where $Z_{M,a,G}>0$ is a positive constant. 
Set $C=(t_1,\ldots,t_p)\in (G/\mathrm{Ad})^p$, when $p\ge 1,$ and $C=\emptyset$ when $p=0.$ The \emph{Yang--Mills measure} on $(M,a)$ with structure group $G$ and boundary conditions
%\footnote{The measure is also called the \emph{conditional Yang--Mills measure}: in fact, conditions can be put on any set of disjoint simple loops in the map, not only the ones representing the boundary of the underlying surface. For more details, see for instance \S 1.5.2 of \cite{Lev3}.}
 $C$ is the probability measure $\YM_{M,C,a,G}$ on $(\mathcal{M}(\mathrm{P}(M),G),\mathcal{C})$ with density\footnote{This measure seems to have  been introduced first in the physics literature in \cite{Mig}, see also \cite{MenO} and \cite{Wit} where it was used to 
compute symplectic volumes of flat connexions. Statistical physics models with heat kernel weight in lattice gauge 
theories, or continuous spin systems, also bear the name of \emph{Villain model},  see for instance \cite{OstSeil}. The term "discrete Yang-Mills measure" might have been introduced and formally defined in \cite[Sect.  2]{Sen6}. }
\begin{equation}
\frac{1}{Z_{M,a,G}(C)}\prod_{f\in \Fbb} p_{a_f}(h_{\pl f}) \label{------eq:DS form}
\end{equation}
with respect to $\mathrm{U}_{M,C,G}$, where $Z_{M,a,G}(C)>0$ is a constant, equal to $Z_{M,a,G}$ if $p=0,$ and defining a function of $t_1,\ldots t_p$ when $p\ge 1.$ Since the heat kernel defines  a functions on $G$ which are central and invariant by 
inversion, 
each term of 
the above product depends neither on the base point, nor on the orientation for the  boundary of the faces, and defines therefore a well defined Wilson loop.  Using the semigroup property leads to the following 
elementary but remarkable lemma.

\begin{lem}[\cite{Sen0,Lev2}] The constant 
\[
Z_{M,a,G}(t_1,\ldots,t_p)= \int_{\mathcal{M}(\mathrm{P}(M),G)} \prod_{f\in \Fbb} p_{a_f}(h_{\pl f})\mathrm{U}_{M,(t_1,\ldots,t_p),G}(dh)
\]
and \[Z_{M,a,G}=\int_{\mathcal{M}(\mathrm{P}(M),G)} \prod_{f\in \Fbb} p_{a_f}(h_{\pl f})\mathrm{U}_{M,G}(dh)\]
depend only on the genus $g$ of $\Sigma$, the total area $T=\sum_{a\in \Fbb}a_f$ and in the first case on the boundary conditions $C=(t_1,\ldots,t_p)$ up to reordering.  When $M$ is not closed, $Z_{M,a,G}=1.$  \end{lem}
This constant is called the \emph{partition function}. We  denote it by $Z_{(g,p),T,G}(t_1,\ldots,t_p),$ or $Z_{g,T,G}$ when $p=0,$ and  we drop the subscript $G$ if there is no ambiguity on the group. When $h\in G^p$, we  also write $Z_{(g,p),T}(h_1,\ldots,h_p)$ for $Z_{(g,p),T}([h_1],\ldots,[h_p])$ where $[g]\in G/\mathrm{Ad}$ denotes the conjugacy class of an element $g\in G.$ We  give  explicit expressions using representation theory in section \ref{sec:Pf Charact}, and an example of discrete Yang--Mills measure in section \ref{sec:example}. %This explicit expression will also help us understand the behaviour of partition functions with boundary conditions under surgery, see Prop. \ref{prop:gluing1} and \ref{prop:gluing2}.

When $p\ge 1,$ the Yang--Mills measure with or without boundary conditions  are related by the disintegration formula
{\begin{equation}
 Z_{M,a,G}\YM_{M,a,G}= \int_{(G/\mathrm{Ad})^p}  Z_{M,a,G}(t_1,\ldots,t_p)\YM_{M,(t_1,\ldots,t_p),a} dt_1\ldots dt_p.\label{-----eq: Disint YM BD}
\end{equation}}

When  $(M',a')$ is another map with genus $g$ and volume $T,$ we say that $(M',a')$ is finer than 
$(M,a)$ if any face $f$ of $M$ is a disjoint union of faces $f_1,\ldots f_k$ of $M'$ with $a_f=a'_{f_1}+\cdots +a'_{f_k}.$ In this case, the restriction of $h\in\mathcal{M}(\mathrm{P}(M'),G)$ to the edges of $M$ defines an element $\mathscr{R}_M^{M'}(h)$ of $\mathcal{M}(\mathrm{P}(M),G)$. We have the following lemma.

\begin{lem}[\cite{Sen0,Lev2}] \label{_____Lem: Compat}When  $C$ is a set of boundary conditions on $\Sigma$ and $(M',a')$ is finer than $(M,a)$, then
$$(\mathscr{R}^{M'}_{M})_* (\YM_{M',C,a',G})= \YM_{M,C,a,G}.$$
In particular if $\ell\in \mathrm{L}(M)$, the random variable $W_\ell$ has same law under $\YM_{M',C,a',G}$ and $\YM_{M,C,a,G}.$
\end{lem}

We further need to  define  the Yang--Mills measures with  a constraint along a single  loop, together with a disintegration formula. Assume that  $M$ is closed,   $\ell\in \Ld(M)$ is a simple loop and $t\in G/\mathrm{Ad}$. Denoting by $\ell= e_{0,1}\ldots e_{0,n}$ the edge decomposition of $\ell$ in $M,$  and labeling $e_1,\ldots,e_m$ the other edges of $M,$  the \emph{Yang--Mills  on} $\mathcal{M}(\mathrm{P}(M),G)$ with structure group $G$ and  \emph{constraint} $C_{\ell\mapsto t}$ is the  probability measure $\YM_{M,C_{\ell\mapsto t},a,G}$ on $(\mathcal{M}(\mathrm{P}(M),G),\mathcal{C})$ with density
%
%\footnote{This measure appeared in the physics literature in \cite{Mig}, see also \cite{MenO} and \cite{Wit} where it was used to 
%compute symplectic volumes of flat connexions. Weighting with the heat kernel in gauge 
%theories  
%or continuous spin systems also bears the name of \emph{Villain action},  see for instance \cite{OstSeil}.  }
\begin{equation}
\frac{1}{Z_{M,a,G}(\ell; t)}\prod_{f\in \Fbb} p_{a_f}(h_{\pl f}) \label{------eq:DS form}
\end{equation}
with respect to $\mathrm{U}_{M,C_{\ell\mapsto t},G}$, where $Z_{M,a,G}(\ell;t)>0$ is a positive constant.

%\begin{prop}[\cite{Lev2,Lev3}]
%Let $(M,a)$ be an area-weighted map embedded in a compact surface $\Sigma$ with genus $g\geq 0$. Let $\ell$ be a loop in $M$, and denote by $\YM_{M,\ell,t,a,G}$ the Yang--Mills measure on $(M,a)$ with the condition that $[H_\ell]=t\in G/\mathrm{Ad}$. We have
%\begin{equation}\label{eq:decondition}
%Z_{M,a,G}\YM_{M,a,G}=\int_{G/\mathrm{Ad}} \YM_{M,\ell, t,a,G}Z_{M,\ell,t,a,G}dt.
%\end{equation}
%\end{prop}

\begin{prop}[\cite{Lev2,Lev3}]
Let $(M,a)$ be a closed, area-weighted map, and let  $\ell\in \Ld(M)$ be a simple loop. Then for any $t\in G/\mathrm{Ad},a\in \Delta_M(T),$ 
\begin{equation}\label{eq:decondition}
Z_{M,a,G}\YM_{M,a,G}=\int_{G/\mathrm{Ad}}Z_{M,a,G}(\ell;t) \YM_{M,C_{\ell\mapsto t},a,G}dt.
\end{equation}
\end{prop}

\begin{rmk} Remarkably, the Yang--Mills measure with constraints is related to the Yang--Mills  measure with boundary conditions, see for instance Theorem \ref{thm:split_separ} below. This relation plays the role of a Markov property. 
The holonomy fields under the Yang--Mills measure  can be furthermore understood as an example of a two-dimensional Markovian holonomy fields as defined in \cite{Lev2}.
\end{rmk}

\subsection{Yang--Mills holonomy field}

The  compatibility relation of Lemma \ref{_____Lem: Compat} suggests that the  measures considered can be obtained as the image of a single measure on a larger probability space. This has indeed been achieved in \cite{Lev3} and 
was then generalised in \cite{Lev2}, leading to the notion of (continuous) Markovian holonomy fields, allowing to consider a very large family of loops at once. Though it is not crucial to our argument, we recall here their definition, as it 
allows to reformulate our main results in a unified continuous model.  A different rigorous continuous approach  has been given earlier by \cite{Dri, Sen4, Sen0}. It relies on stochastic analysis of the white noise on the plane and the 
formula \eqref{------eq:DS form} is obtained as a consequence of the 
construction and it is  called the \emph{Driver--Sengupta formula}.  The random holonomy field can be  understood as the parallel transport of a random connection with curvature given by a white noise.  Recently, 
yet another continuous construction, defining a random connection one form has been given in \cite{Che} recovering the latter formula.

Let $\Sigma$ be a closed, connected, orientable surface, an open disc of $\R^2$ or $\R^2$ itself, endowed with an area measure $\mathrm{vol}$. For any $\ell\in \Ld(\Sigma)$, we use the same notation as in \eqref{------eq: Wilson loop}, for the random variable
\[
W_\ell= \tr(H_\ell).
\]
Furthermore, whenever $h\in \mathcal{M}(\mathrm{P}(\Sigma),G)$, the restriction of $h$ to a topological map $M$ on $\Sigma$ defines an element $\mathscr{R}_{M}(h)$ of $\mathcal{M}(\mathrm{P}(M),G)$. Let us finally mention that the space $\mathcal{M}(\mathrm{P}(\Sigma),G)$ can be endowed with two sigma-fields $\mathcal{C}$ and $\mathcal{J}$ defined as in the discrete case. The continuous extension to the discrete Yang--Mills measure is provided by the following result.

%associated to the canonical process $(h_\g, \g\in \mathrm{P}(\Sigma))$
%call Denote by 
%
%$\mathrm{P}(\Sigma)$ the set of continuous maps $[0,1]\to \Sigma$  
%In this text we shall assume that  with a measure of area of total mass $T>0.$ 

\bth[\cite{Lev2}]  Assume that $\Sigma$ is a closed, connected, orientable two dimensional Riemannian manifold, an open disc of $\R^2$ or $\R^2$,  with  area measure $\mathrm{vol}$. For any set $C=(t_1,\ldots,t_p)$ of boundary conditions, there exist   
probability measures $\YM_{\Sigma,G}$ and $\YM_{\Sigma,C,G}$ on
$(\mathcal{M}(\mathrm{P}(\Sigma),G),\mathcal{C})$, such that 
\begin{enumerate}
\item  for any rectifiable map $M$ on $\Sigma$,
\[
({\mathscr{R}_{M}})_*(\YM_{\Sigma,G})=\YM_{M,\mathrm{vol},G}\text{ and }({\mathscr{R}_{M}})_*(\YM_{\Sigma,C,G})=\YM_{M,C,\mathrm{vol},G},
\]
\item \[Z_{g,T} \YM_{\Sigma,G}= \int_{(G/\mathrm{Ad})^p}  Z_{(g,p),T}(t_1,\ldots,t_p)\YM_{M,(t_1,\ldots,t_p),\mathrm{vol},G}dt_1\ldots dt_p, \]
 \item if $(\g_n)_{n\ge 0}$ is a sequence of $\mathrm{P}(\Sigma)$  and $\g\in \mathrm{P}(\Sigma)$ with $\underline{\g_n}=\underline{\g}, \overline{\g_n}=\overline{\g},\forall n\ge 0$ and $d(\g_n,\g )\to 0$ as $n\to \infty,$ then,  under 
 $\YM_{\Sigma,G,C},$
 $(H_{\g_n})_{n\ge 0} $ converges in distribution towards $H_\g.$
\end{enumerate}
Moreover, if $\Psi$ is a diffeomorphism of $\Sigma$ preserving $\mathrm{vol}$,  then under $\YM_{\Sigma,C,G}$ and $\YM_{\Sigma,G},$ $W_\ell$ has same law as $W_{\Psi(\ell)}.$
\eth
\begin{prop}[\cite{Lev2}]  When $\ell\in \mathrm{L}(\Sigma)$ is a simple loop of a closed, connected, orientable two dimensional Riemannian manifold  $\Sigma$ with area measure $\mathrm{vol}$, and $t\in G/\mathrm{Ad},$ there is a measure $\YM_{\Sigma,C_{\ell\mapsto t},G}$ on 
$(\mathcal{M}(\mathrm{P}(\Sigma),G),\mathcal{C})$ such that 
\begin{enumerate}
\item  For any rectifiable map $M$ on $\Sigma$ such that $\ell$ is a drawing of a loop of $M$,
\[
({\mathscr{R}_{M}})_*(\YM_{\Sigma,C_{\ell\mapsto t},G})=\YM_{M,C_{\ell\mapsto t},\mathrm{vol},G}.
\]
Moreover the constant $Z_{M,a,G}(\ell;t)$ only depends on $\ell\in\Ld^2(\Sigma)$ and $t$. We denote it  by $Z_{\Sigma,G}(\ell;t).$
% \item If $(\g_n)_{n\ge 0}$ is a sequence of $\mathrm{P}(\Sigma)$  and $\g\in \mathrm{P}(\Sigma)$ with $\underline{\g_n}=\underline{\g}, \overline{\g_n}=\overline{\g},\forall n\ge 0$ and $d(\g_n,\g )\to 0$ as $n\to \infty,$ then,  under 
% $\YM_{\Sigma,G,C},$
% $(H_{\g_n})_{n\ge 0} $ converges in distribution towards $H_\g.$
 \item If $\Sigma$ has genus $g$ and total volume $T$,
   \begin{equation}\label{-----eq:deconditionCont}
Z_{g,T}\YM_{\Sigma}=\int_{G/\mathrm{Ad}}Z_{\Sigma,G}(\ell;t) \YM_{\Sigma,C_{\ell\mapsto t},a,G}dt.
\end{equation} 
\end{enumerate}
\end{prop}
The random process $(H_\g,\g \in \mathrm{P}(\Sigma))$ with distribution given by $\YM_{\Sigma,C,G}$ is called the \emph{Yang--Mills holonomy field} on $\Sigma$. We are primarily interested in the random variables $W_\ell$ for $\ell\in  
\mathrm{L}(\Sigma)$, with structure group a classical group $G_N$ of size $N$.

\subsection{Master fields, conjectures and main results}

Following the physics literature, I. M. Singer raised in \cite{Sin}  the following question, that  we will reformulate  here in a slightly\footnote{the specific family of loops was not specified in \cite{Sin}.} modified form as a  conjecture.

\begin{conj}  \label{_________Conj: Singer} Let $\Sigma$ be a  closed, connected, orientable, two dimensional, Riemannian manifold  $\Sigma,$ or  the Euclidean plane $\R^2,$ or a disc of $\R^2$.  Assume that $G_N$ is a classical group of size $N$ with metric given by \eqref{------eq:inner P LieAlg}. Then for  any loop $\ell\in \Ld(\Sigma) $, there 
is a constant $\Phi_\Sigma(\ell)$ such that, under $\YM_{\Sigma,G},$ 
\begin{equation}
W_\ell \to \Phi_\Sigma(\ell) \text{ in probability as }N\to \infty.\label{------eq: Conv PP}
\end{equation}
If $\Psi$ is a diffeomorphism of $\Sigma$ preserving its volume form,
$$\Phi_\Sigma(\Psi(\ell))=\Phi_\Sigma(\ell).$$
\end{conj}
The limit function
$$\Phi_\Sigma: \mathrm{L}(\Sigma)\to \C$$
is called a \emph{master field}. This conjecture has  been partly proved for the plane, working with a  smaller class of loops in \cite{Xu,AS}, for unitary groups. In \cite{Lev} , it was 
 simultaneously proved for the plane and for all above group series. Recently another argument using Makeenko--Migdal equations was also given in \cite{Hal2}. 
\begin{rmk}  For any loop, $\ell\in \Ld(\Sigma),$ since $H_\ell$ is a unitary matrix, $|\tr(H_\ell)|\le 1$ and the convergence in probability \eqref{------eq: Conv PP} is equivalent to 
\begin{equation*}
\E[|\tr(H_\ell)-\Phi_{\Sigma}(\ell)|]\to 0 \tag{*}
\end{equation*}
as $N\to \infty.$ When we say that \eqref{------eq: Conv PP} holds uniformly in the set of area vectors $a$, it is equivalent to the uniformity of the convergence (*)  in $a.$ 
To show (*), it  is sufficient to show
$$\E[\tr(H_\ell)]\to \Phi_{\Sigma}(\ell) \text{ and }\mathrm{Var}(\tr(H_\ell))=\E[\tr(H_\ell)\overline{\tr(H_\ell)}]-|\E[\tr(H_\ell)]|^2\to 0.$$
\end{rmk}
\begin{rmk} The linear extension of a master  field $\Phi_\Sigma$ to $\C[\mathrm{L}(\Sigma)] $ comes automatically  with a structure of non-commutative probability space 
\cite{Lev,DN}.  In the case of the plane, this non-commutative distribution  can be characterised using free probability (\cite{Lev} and \cite{CDG}). In that case, one can 
recover the 
distribution of a free Brownian motion from the master field (\cite{Bia,Lev,CDG}). 
\end{rmk}

\begin{thm}[\cite{Xu,AS,Lev} and \cite{DN}]\label{__THM: Plane}
The conjecture \ref{_________Conj: Singer} holds true when $\Sigma$ is an open disc of the plane or a sphere of total area $T>0$.  Consider  $\ell\in\Ld(M)$ where $(M,a)$ is  a closed, area-weighted map with \emph{fixed} total area $T.$ Then  
 under $\YM_{M,a},$ 
$$W_\ell\to \Phi_{M,a}(\ell) \text{ in probability},$$
where the right-hand side is deterministic  and depends continuously on $a$, over $a\in \R_+^{\Fbb}$ in the case of the plane and $a\in \Delta_M(T)$  in the case of the sphere.
\end{thm}

The 
work  \cite{Lev} was the first to 
show rigorously  that the master field $\Phi_{\R^2}$  satisfies a set of differential equations named after Makeenko--Migdal equations, that appeared earlier in the physics papers 
\cite{MM,KK}. In 
\cite{DN}, the conjecture was proved for unitary 
groups, in the 
case of 
the sphere. 

\begin{ex}[\cite{Xu,AS,Lev,DN}] Assume that $\ell\in \Ld(\R^2)$ is simple and encloses an area $t.$
For $n\ge 1,$
$$\mu_t(n):=\Phi_{\R^2}(\ell^n)=\frac{e^{-\frac{nt}2 }}{n}\sum_{m=0}^{n-1} \frac{(-nt)^m}{m!} {n \choose m+1}.$$
 Denote by $\Sbb^2_T $ the two-dimensional Euclidean sphere with total volume $T$.  When\footnote{An the expression for any $T>0$  was proved in \cite[Thm 2.4]{DN}. }  $T\le \pi^2$ and  $\ell\in 
\Ld(\Sbb_T^2)$ 
is simple, cutting  the sphere into two domains of area $t$ and $T-t,$
$$\mu_{t,T}(n):=\Phi_{\Sbb^2_T}(\ell^n)= \frac{1}{n\sigma} J_1(2 n \sigma)=\int_{-2}^2\exp(i n\sigma x) \frac{\sqrt{4-x^2} dx}{\pi},  $$
where $\sigma= \sqrt{\frac{t(T-t)}{T}}$ and 
$$J_{1}(x)=\sum_{m\ge 0} \frac{(-1)^m}{m!(m+1)!}\left(\frac x 2 \right)^{2m+1}$$
is a Bessel function of the first kind.
 \end{ex}
\begin{rmk}  \begin{enumerate}
\item Note that  when $n\ge 1$ and $T\le \pi^2$ are fixed, $\mu_t(n)$ and $\mu_{t,T}(n)$ are different functions of $t$. This is also true \cite{DN} when $T>\pi^2,$ though 
the expression of $\mu_{t,T}$ is different.\footnote{In the physics literature, the regimes $T\le \pi^2$ and $T>\pi^2$ are  respectively 
called the weak and the strong regimes \cite{DK}.} Therefore, 
when $\ell$ is a simple loop enclosing a disc of area $t$, 
the 
expression of  the master field is not the same when the surface in which $\ell$ is drawn is the plane or the sphere.
\item Let us highlight nonetheless two relations between the master field on the sphere and on the plane.  On the 
one hand, it can be shown  \cite{DN}  that
\begin{equation}
\lim_{T\to \infty}\mu_{t,T}(n)= \mu_{t}(n),\forall n\ge 1. \label{------eq:Strong Sphere to plane}
\end{equation}
On the other hand, it follows from dominated convergence that for all $t\ge 0$
$$\lim_{k\to \infty}\mu_{\frac{t}{k^2}}(nk)= J_1(2 n \sqrt{t}). $$
Therefore, for all $0\le t\le T\le \pi^2,$
\begin{equation}
\lim_{k\to \infty}\mu_{\frac{\sigma^2}{k^2}}(nk)=\mu_{t,T}(n),\forall n\ge 1.\label{------eq:Conv Plane to weak sphere}
\end{equation}
\item The sequences $\mu_t(n)$ and $\mu_{t,T}(n),n\ge 1$ are moment sequences of  measures $\mu_t$ and $\mu_{t,T}$ on the unit circle,  associated to a time marginal of the 
free Brownian motion \cite{Bia} and of the free brownian bridge \cite{DN}. 
Since both $\mu_t$ and $\mu_{t,T}$ are invariant by complex conjugation, \eqref{------eq:Strong Sphere to plane} and \eqref{------eq:Conv Plane to weak sphere} imply the weak 
convergences
$$\mu_{t,T}\to \mu_t  \text{ as  }T\to \infty$$
and for any $t\le T\le \pi^2,$
$$\mu^k_{\frac{\sigma^2}{k^2}}\to \mu_{t,T} \text{ as }k\to\infty, $$
where for any measure on $\nu$ on the unit circle, $\nu^k$ denotes the push forward of $\nu$ by $z\mapsto z^k.$ 
\end{enumerate}
\end{rmk}

The conjecture \ref{_________Conj: Singer} remains open for general surfaces. Though, using the Makeenko--Migdal equations on surfaces proved in \cite{DGHK},  it was realised in  
\cite{DN,Hal2} that it is sometimes enough to show the convergence for a restricted family of 
loops. This idea was exploited for general surfaces in \cite{Hal2} yielding the following theorem.\footnote{Another similar result is obtained in \cite{Hal2}, where only an 
assumption on simple loops is made. For $g\ge 1,$ the conclusion is then weaker and holds for loops with constrained area vector.} Let us say that a loop $\ell$ of a map $M$ is included 
in a disc $U$  if its drawing is included in an open, contractible set $U$ of $\Sigma$.

\begin{thm}[\cite{Hal2}] \label{__THM:Hall}Consider $G_N=\U(N)$. Assume that  whenever  $\ell=s^n$ with $n\ge 0$ and $s$ is a simple loop included in a disc, of a closed, area-weighted 
topological map\footnote{Recall that  we assumed here that, by convention, a closed topological map can be embedded in a closed, connected, orientable surface.} $(M,a)$, under $\YM_{M,a,G_N},$  the random variable $W_\ell$ converges in 
probability towards a constant as $N\to \infty$.  
 Then this also holds true for $W_\ell$  for any combinatorial loop $\ell$  included in a disc.

\end{thm}

B. Hall conjectured further that the above assumption can be removed. 

\begin{conj}[\cite{Hal2}] Consider $G_N=\U(N)$.  Whenever  $\ell=s^n$ with $n\ge 0$ and $s$ is a simple loop  included in a disc of a closed, area-weighted topological map $(M,a)$, 
$W_\ell$ converges in probability towards a constant as $N\to 
\infty$, under $\YM_{M,a,G_N}$.   
\end{conj}

Our main result implies that this conjecture holds true for any closed surface $\Sigma$ and any group series of classical groups.  We simultaneously prove  Theorem \ref{__THM:Hall}
 without using the Makeenko--Migdal 
equations. Instead we shall use the convergence of the partition of function; a result proved by the first author in \cite{DN} for unitary groups that we recall and 
generalise to other group series in section \ref{sec:PF}.  Remarkably the limit in the above conjecture is given in terms of the master field on the plane as follows. 

Assume that $\ell$ is a loop of an area 
weighted  map $(M,a)$ included in a closed disc $U$ of $\Sigma.$  Considering only vertices, edges and area-weighted faces of $M$ which are mapped by the embedding of $M$ in $U,$ and replacing all 
faces intersecting 
$\Sigma\setminus U$ by a single marked face yields\footnote{For instance, in Fig. \ref{fig:exemple}, we consider a map $M$ with $3$ faces embedded in the torus, and in Fig. \ref{fig:exemple2}, an additional embedded disc $U$.  The resulting map $\tilde M_U$  with $2$ vertices, two edges and $2$ bounded faces $F_1,F_2,$ is drawn in Fig. \ref{fig:exemple3}.} an area-weighted map $(\tilde M_{U},a_{U})$ of $\R^2$.  
%Assigning the area $a_f$ to each face  $f$ included in $U$  yields an area-weighted  map $(\tilde M_{U},a_{U})$ of $\R^2.$
We then denote  by the same symbol the loop of $\tilde M_U$ obtained by concatenating the edges of $\ell.$

\begin{thm} \label{__THM:Disc dYM}Consider an area-weighted topological map $(M,a)$ on a closed surface $\Sigma$ of genus $g\ge 1$ and total volume $T.$ Assume that $\ell$ is a loop  of $(M,a)$ included in a disc $U$ of $M.$
%$ with 
%$\|a_U\|_1 <\|a\|_1.$
 Then for any classical group $G_N$ of size $N$, under $\YM_{M,a,G_N},$
\[
W_\ell \to \Phi_{\tilde M_U,a_U}( \ell)\text{ in probability as } N\to\infty,
\]
uniformly in $a\in \Delta_M(T).$
\end{thm} 

\begin{ex} In particular, if $\ell=s^n$ where $s$  is a simple contractible loop enclosing an area $t\in( 0,T],$ then  
$$W_\ell\to \frac{e^{-\frac {nt }2}}{n}\sum_{m=0}^{n-1} \frac{(-nt)^m}{m!} {n \choose m+1} \text{ in probability as } N\to\infty. $$
\end{ex}
Using a continuous construction and uniformity estimates obtained in \cite{Lev2}, this result can be generalised  as follows to a continuous setting allowing to consider a much wider family of loops.
\begin{thm} \label{__THM: Disc YMCont}Assume that $\Sigma$ is a closed, connected, orientable Riemann surface with Riemannian volume $\mathrm{vol}$,   $\Psi:U\to D_R=\{x\in \R^2:x_1^2+x_2^2<R\}$ is a diffeomorphism, where  $\pi R^2<T,$  $U$ is an 
open set of $\Sigma$ and 
$\Psi_*(\mathrm{vol}_{|U})$ is the Lebesgue measure. Then, for any $\ell\in \Ld(U),$ and any classical group $G_N$ of size $N$, under $\YM_{\Sigma,G_N},$ 
\[
W_\ell\to \Phi_{\R^2} (\Psi(\ell)) \text{ in probability as } N\to\infty.
\]
Assuming furthermore that  $\Psi^{-1}:D_{R}\to  U$ can be extented continuously to $\overline{D}_R,$ with piecewise continuous derivatives, for any loop $\ell\in \Ld(\overline{D}_R),$
\[
W_{\Psi^{-1}(\ell)}\to \Phi_{\R^2} (\ell) \text{ in probability as } N\to\infty.
\]
 \end{thm}

Note  that a simple loop is included in a disc if and only if\footnote{Let us recall an argument for the converse statement.  Consider the surface $\Sigma'$ obtained by cutting $\Sigma$ along $\ell.$  It has either one or two connected components. If $\Sigma'$ is connected, then, according to \cite[Sect. 6.3.1]{Sti} or \cite[Sect. 1.3.1]{FM},  we can assume that $\ell$ is one of the generators of the fundamental group of $\Sigma.$ Hence $\ell$ is not contractible. If $\Sigma'$ has two connected components $\Sigma_1$ and $\Sigma_2,$ we want to show that both have 
zero genus.  Would both of them have positive genus, then $\ell$ would not be contractible in either of them.  By  Seifert--van Kampen theorem, it would follow that $\ell$ is not contractible in $\Sigma$.} it is contractible. Our second result below is concerned with simple loops $\ell$ which are not contractible.

\begin{thm} \label{__THM:simple non-contrac} Assume that $\Sigma$ is a closed, connected Riemann surface of genus $g\geq 1$, with Riemannian volume $\mathrm{vol}$.   Then, for all  simple, non-contractible loop $\ell$ on $\Sigma$, for any  $k\in \Z^*$ and any classical group $G_N$ of size $N$, under $\YM_{\Sigma,G_N}$,
\[
W_{\ell^k} \to0 \text{ in probability as } N\to\infty. 
\]
\end{thm}

%\begin{thm} Assume that $G=U(N)$ and   $\ell$ is a simple non-separating loop of an area-weighted map $(M,a)$.   Then, for any integer $n\not=0,$ under $\YM_{M,a},$ as $N\to \infty,$
%$$W_\ell \to0 \text{ in probability}. $$
%\end{thm}

Our result covers all simple loops on closed, connected, orientable surfaces. We further believe it holds true for non-orientable surfaces. In a sequel to the current work \cite{DL}, we 
investigate the 
master field question for all loops with self-intersections.

\subsubsection{Example and idea of proof  for Theorem \ref{__THM:Disc dYM}}\label{sec:example}

In this section we illustrate the key objects, namely the discrete Yang--Mills measure, the Wilson loops and the master field, with an explicit example on a torus of area $T$. We then give an idea of the proof of Theorem \ref{__THM:Disc dYM} for this example. 

\begin{figure}[!h]
\begin{center}
\includegraphics[scale=1]{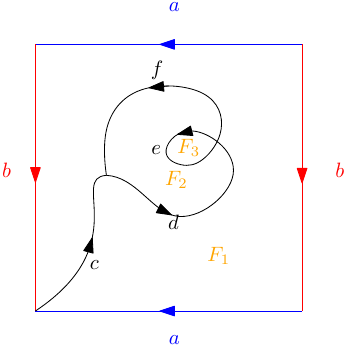}
\caption{An example of area-weighted map on a torus.}\label{fig:exemple}
\end{center}
\end{figure}

Consider the area-weighted map $(M,a)$ in Fig. \ref{fig:exemple}. It has $3$ vertices, $6$ edges labelled $a,b,c,d,e,f$, and $3$ faces whose boundaries are described clockwise by $\partial F_1=b^{-1}a^{-1}bacdfc^{-1}$, $\partial F_2=f^{-1}ed^{-1}$ and 
$\partial F_3 = e^{-1}$. Denote by $a_i$ the weight of the face $F_i$, $i\in\{1,2,3\}$. Denote respectively by $x_1,\ldots,x_6$ the values of the  edge variables $h_a,\ldots,h_f$. The unconditional Yang--Mills measure on $(M,a)$ is given by
\begin{align}
\begin{split}
d\YM_{M,a,G}(x_1,\ldots,x_6)=&\frac{1}{Z_{M,a,G}}p_{a_1}(x_3^{-1}x_6x_4x_3[x_1,x_2])\\
&\times p_{a_2}(x_4^{-1}x_5x_6^{-1})p_{a_3}(x_5^{-1})dx_1\cdots dx_6.
\end{split}
\end{align}
The unconstrained partition function $Z_{M,a,G}$ is given by
\begin{equation}
Z_{M,a,G}=Z_{1,T,G}=\int_{G^2}p_{T}([x,y])dxdy,
\end{equation}
and the partition function $Z_{M,a,G}(b;[y]) $ for the Yang-Mills measure with constraint $C_{b\mapsto [y]},$ where $y\in G,$ is
\begin{equation}
Z_{M,a,G}(b;[y]) =\int_{G}p_{T}([x,zyz^{-1}])dx=\int_{G}p_{T}([x,y])dx.
\end{equation}
Consider the area weighted map $(M^*,a^*)$ with one boundary component, which is identical to $M$ but with  $F_2$ and $F_3$ removed. Then for all $x\in G,$
\begin{align*}
Z_{M^*,a^*,G}(x)&=\int_{G^5}  p_{a_1}(x_3^{-1}x_4x_6 x_3[x_1,x_2]) dx_1dx_2dx_3\delta_{[x]}(dx_4dx_6) \\
&=\int_{G^2}p_{a_1}(x[x_1,x_2])dx_1dx_2= Z_{(2,1),a_1}(x).
\end{align*}
Lastly, instead of the first map above, consider  an area weighted map $(M,a)$  with an additional edge, labeled $t$, bounding a closed disc $U$ and splitting the face $F_1$ into two faces $F_0$ and $F_1$ as in Fig. \ref{fig:exemple2}. 
The discrete Yang--Mills 
measure is then
\begin{align}
\begin{split}
d\YM_{M,a,G}(x_1,\ldots,x_7)=&\frac{1}{Z_{1,T,G}}p_{a_0}(x_3^{-1}x_7x_3[x_1,x_2])p_{a_1}(x_7^{-1}x_6x_4)\\
&\times p_{a_2}(x_4^{-1}x_5x_6^{-1})p_{a_3}(x_5^{-1})dx_1\cdots dx_7,\label{eq:YMex}
\end{split}
\end{align}
with the same partition function as the initial area-weighted map.

\begin{figure}[!h]
\begin{center}
\includegraphics[scale=1]{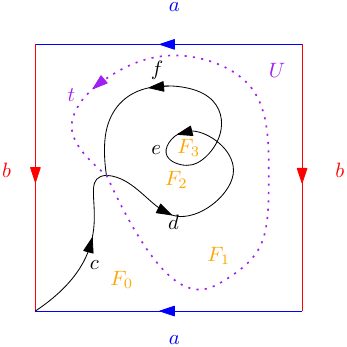}
\caption{An area-weighted map on the torus. The loop $def$ is included in the disc $U$ with boundary $t$.}\label{fig:exemple2}
\end{center}
\end{figure}

Consider the loop $\ell=efg$, included in the disc $U$ whose boundary is $t$. Under $\YM_{M,a,G},$ the Wilson loop $W_\ell$ is a complex random variable whose expectation  is
\begin{align*}
\E[W_\ell]=&\frac{1}{Z_{1,T,G}}\int_{G^7}\tr(x_6x_5x_4)p_{a_0}(x_3^{-1}x_7x_3[x_1,x_2])p_{a_1}(x_7^{-1}x_6x_4)\\
&\hspace{4cm}\times p_{a_2}(x_5^{-1}x_6x_7^{-1})p_{a_3}(x_5^{-1})dx_1\cdots dx_7.
\end{align*}
\begin{figure}[!h]
\begin{center}
\includegraphics[scale=1]{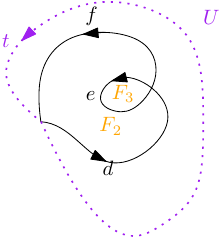}
\caption{The lifting of the loop $\ell$ of Fig. \ref{fig:exemple} in the plane.}\label{fig:exemple3}
\end{center}
\end{figure}
The weighted map $(\tilde M_U,a_u)$ has then only two bounded faces, $F_2$ and $F_3,$ as shown on figure \ref{fig:exemple3}. 
%The lifted loop $\ell$ in $\tilde{M}_U$ is displayed in Fig. \ref{fig:exemple3}. 
 We have, under the corresponding planar Yang--Mills measure $\YM_{\tilde{M}_U,a_U}$,
% &\int_{G^4}\tr(x_6x_5x_4)p_{a_1}(x_7^{-1}x_6x_4)p_{a_2}(x_4^{-1}x_5x_6^{-1})p_{a_3}(x_5^{-1})dx_4\cdots dx_7\\
\begin{align*}
\E[W_\ell]= \int_{G^3}\tr(x_3x_2x_1)p_{a_2}(x_1^{-1}x_2x_3^{-1})p_{a_3}(x_2^{-1})dx_1dx_2dx_3.
\end{align*}
According to Theorem \ref{__THM: Disc YMCont}, the limit of $\E[W_\ell]$ with respect to $\YM_{(M,a)}$ is the same as the one with respect to $\YM_{(\tilde{M}_U,a_U)}$. We can then use the table of the master field on the plane \cite{Lev} to compute $\Phi_{\tilde{M}_U,a_U}(\ell)$, and we obtain that
\[
\Phi_{\tilde{M}_U,a_U}(\ell)=e^{-\frac{a_2}{2}-a_3}(1-a_3).
\]
Let us consider another loop, using that the convergence of Theorem \ref{__THM: Disc YMCont} holds uniformly on  $\Delta_M(T).$ Consider the loop $\g=f^{-1}d^{-1}$ in $M.$   Then under $\YM_{M,a},$
\begin{align*}
\E(W_\g)&= \frac{1}{Z_{1,T,G}}\int_{G^7}\tr((x_6x_4)^{-1})p_{a_0}(x_3^{-1}x_7x_3[x_1,x_2])p_{a_1}(x_7^{-1}x_6x_4)\\
&\hspace{4cm}\times p_{a_2}(x_4^{-1}x_5x_6^{-1})p_{a_3}(x_5^{-1})dx_1\cdots dx_7\\
&= \frac{1}{Z_{1,T,G}}\int_{G^4}\tr((x_6x_4)^{-1})p_{a_0+a_1}(x_6x_4[x_1,x_2]) p_{a_2+a_3}(x_4^{-1}x_6^{-1}) dx_1dx_2dx_4dx_6.
\end{align*}
and 
\begin{align*}
\lim_{a_0,a_1\downarrow 0}\E(W_\g)&= \frac{1}{Z_{1,T,G}}\int_{G^2} \tr([x_1,x_2])p_{T}([x_1,x_2])dx_1dx_2=\E(W_{[a,b]}),
\end{align*}
where the expectation on the right-hand side is with respect to the map with respect to the Yang-Mills measure on the torus of area $T.$ 

According to Theorem \ref{__THM: Disc YMCont} and the above identity, the limit of $\E[W_{\g}]$ with respect to $\YM_{(M,a)}$ is the same as the one of $\E[W_{\g}]$ with respect to $\YM_{(\tilde{M}_U,a_U)}$ as $N\to\infty$ and it holds uniformly in $a\in\Delta_M(T).$ In particular, using the value \cite{Lev} of the planar master field on simple loops, we find
\begin{equation}
\lim_{N\to\infty} \E[W_{[a,b]}]=\lim_{a_0,a_1\downarrow 0,a\in\Delta_M(T)} \lim_{N\to\infty} \E[W_{\g}]=\lim_{a_0,a_1\downarrow 0,a\in\Delta_M(T)}e^{-\frac{a_2+a_3}{2}}=e^{-\frac T 2}.
\end{equation}

\emph{Idea of proof for Theorem  \ref{__THM:Disc dYM} and $\ell$ in figure \ref{fig:exemple}.} Let us give now  an idea of the proof of Theorem \ref{__THM:Disc dYM}  for $\ell=def$ as in figure \ref{fig:exemple}.  Keeping the same notation for the value of the edge variables, the Yang-Mills measure in the area weighted map  $(\tilde{M}_U,a_U)$ is 
\begin{align}
\begin{split}
d\YM_{\tilde{M}_U,a_U,G}(x_4,x_5,x_6)=&p_{a_2}(x_4^{-1}x_5x_6^{-1})p_{a_3}(x_5^{-1})dx_4dx_5 dx_6.\label{eq:YMexcut}
\end{split}
\end{align}
From the expression \eqref{eq:YMex} for the discrete Yang-Mills measure in the map $(M,a)$ of figure \ref{fig:exemple2}, it follows that under $\YM_{M,a,G}$, the random variable $(h_c,h_e,h_f)$  is absolutely continuous  with respect to $\YM_{\tilde{M}_U,a_U,G}$ with density
   \begin{align}
F( x_4,x_5,x_6)&=\frac{1}{Z_{1,T,G}}\int_{G^{4}} p_{a_0}(x_3[x_1,x_2]x_3^{-1}x_7)p_{a_1}(x_7^{-1}x_6x_4)dx_1dx_2 dx_3dx_7\label{eq: Abs Cont Ex}\\
&=\frac{Z_{(1,1),a_0+a_1,G}(x_6x_4)}{Z_{1,T,G}}.
\end{align}
 We conclude that for any $\varepsilon>0,$ $\PP_{\YM_{M,a,G}}( |W_{\ell}-\Phi_{\tilde{M}_U,a_U}(\ell)| >\varepsilon)$ is bounded by
\[ \frac{\|Z_{(1,1),a_0+a_1,G}\|_{\mathrm{L}^{\infty}(G)}}{Z_{1,T,G}}\PP_{\YM_{\tilde{M}_U,a_U,G}}( |W_{\ell}-\Phi_{\tilde{M}_U,a_U}(\ell)| >\varepsilon).  \]
By \cite{Lev}, the last probability vanishes when $G=G_N$ is as in section \ref{sec:Intro Class Gp} and $N\to \infty.$  To conclude, it is enough to bound the ratio uniformly in $N$. This will be a consequence of the observation that 
\begin{equation}
\|Z_{(1,1),a_0+a_1,G_N}\|_{\mathrm{L}^{\infty}(G)}=Z_{(1,1),a_0+a_1,G_N}(1)=Z_{1,a_0+a_1,G_N}
\end{equation}
combined with Theorem \ref{__THM:PF}.  The 
full proof of Theorem
\ref{__THM:Disc dYM} for a fixed area\footnote{It can be shown that the ratio $\frac{\sup_{0<s<T}Z_{1,s,G_N}}{Z_{1,T,G_N}}=\frac{Z_{1,0,G_N}}{Z_{1,0,G_N}}$ diverges as $N\to\infty,$ and  an additional argument 
is required to prove the uniform convergence in $\Delta_M(T)$.} follows this 
argument and is given in section \ref{sec: Loops in a disc}.

\section{Asymptotics of partition functions}\label{sec:PF}

%It can be proved (see \cite{Lev3} for instance) that the Yang--Mills measure is invariant by subdivision. In particular, it appears that the partition function is independent from the choice of a graph, and from now on we will denote it by $Z_{T,G}$, or even $Z_T$ when the group $G$ is unambiguous.
Our argument  relies on the convergence of partition functions of closed orientable surfaces, for classical groups, generalising results of  \cite{Gur} and of the second author 
\cite{Lem}.
This result was mentioned without 
proof in \cite{Rus} for $g>1$ and in \cite{Dou} for $g=1$ and $G=\U(N)$ or $\SU(N)$. For $g\ge 2$, for all group series except the unitary group series, it is a simple consequence 
of 
\cite{Gur}. For $g=1$ it is proved in \cite{Lem} for $\U(N)$ and $\SU(N)$. For all classical groups series, the limit  involves the Jacobi theta function\footnote{The Jacobi theta function only appears for the unitary 
group series.} and the Euler phi 
function, both defined for $q\in\C$ such that $|q|<1$, with
\[
\theta(q)= \sum_{n\in\Z} q^{n^2} \ \text{ and } \ \phi(q)=\prod_{m=1}^\infty (1-q^m).
\]
%
%\begin{thm}\label{__THM:PF0}
%For any $T\geq 0$, $(N,g)\in(\N^*)^2$, let us denote by $Z_{g,T,\U(N)}$ (resp. $Z_{g,T,\SU(N)}$) the Yang--Mills partition function on an orientable compact surface of genus $g$ and total area $T$ with structure group $\U(N)$ (resp. $\SU(N)$). Set $q_T=e^{-\frac{T}{2}}$.
%\begin{enumerate}
%\item If $g=1$ and $T>0$, then the following convergences hold:
%\[
%\lim_{N\to\infty} Z_{T,U(N)} = \frac{\theta(q_T)}{\phi(q_T)^2} \text{ and } \lim_{N\to\infty} Z_{T,SU(N)} = \frac{1}{\phi(q_T)^2}.
%\]
%\item If $g\geq 2$ and $T>0$, then the following convergences hold:
%\[
%\lim_{N\to\infty} Z_{T,U(N)} = \theta(q_T) \text{ and } \lim_{N\to\infty} Z_{T,SU(N)} = 1.
%\]
%Moreover, if $g\geq 2$ and $T=0$ then we have
%\[
%\lim_{N\to\infty} Z_{T,SU(N)}=1.
%\]
%\end{enumerate}
%\end{thm}

%We will extend this result to the other classical groups. 
Let $r\geq 1$ be an integer, we denote respectively by $\tilde A_r, A_r, B_r$, $C_r$, $D_r$ the type of $\U(r)$, $\SU(r+1)$, $\SO(2r+1)$, $\Sp(r)$ or $\SO(2r)$, in reference to the type of their root system (see \cite{BtD} for instance). Note that we will alternatively use $r$ as the rank of the root system or $N$ as the size of the classical group, and their relations will be implicit: for instance, if we consider $\SU(N)$, it will be of type $A_r$ with $r=N-1$. Conversely, if we consider the classical group of type $D_r$, it will be $\SO(2r)$. This change of index may be confusing at first, but it will be helpful in order to deal simultaneously with all considered root systems.  Considering limits when $r\to\infty$ and when $N\to\infty$ is  equivalent.

%\begin{thm}\label{__THM:PF}
%For any $T>0$, $(N,g)\in(\N^*)^2$ and any type $X_N$, with $X\in\{B,C,D\}$, let us denote by $Z_{g,T,X_N}$ the Yang--Mills partition function on an orientable compact surface of genus $g$ and area $T$ with structure group of type $X_N$ (we will also write $Z_{g,T,N}$ when the type of the structure group is unambiguous). Set $q_T=e^{-\frac{T}{2}}$.
%\begin{enumerate}
%\item If $g=1$ and $T>0$, then $Z_{1,T,X_N}$ converges when $N\to\infty$, and the limits are given by
%\begin{equation}
%\lim_{N\to\infty} Z_{1,T,B_N} = \lim_{N\to\infty} Z_{1,T,C_N}=\frac{1}{\phi(q_T)}
%\end{equation}
%and
%\begin{equation}
%\lim_{N\to\infty} Z_{1,T,D_N} = \frac{1+q_T}{(1-q_T)^2\phi(q_T)}.
%\end{equation}
%\item If $g\geq 2$, then for any $X\in\{A,B,C,D\}$ we have
%\begin{equation}
%\lim_{N\to\infty} Z_{g,T,X_N} = 1.
%\end{equation}
%\end{enumerate}
%\end{thm}

\begin{thm}\label{__THM:PF}
For any $T>0$, $(r,g)\in(\N^*)^2$ and any type $X_r$, with $X\in\{B,C,D\}$, let us denote by $Z_{g,T,X_r}$ the Yang--Mills partition function on an orientable compact surface of genus $g$ and area $T$ with structure group of type $X_r$ (we will also write $Z_{g,T,r}$ when the type of the structure group is unambiguous). Set $q_T=e^{-\frac{T}{2}}$. 
\begin{enumerate}
\item For all $g\ge 1$ and $T>0$, $\lim_{r\to \infty} Z_{g,T,X_r}$ exists and is given by the table \eqref{------eq:PFtable} below. 
\item Moreover if  $g=2$, for all $X\in \{A,B,C,D\},$ $Z_{g,T,X_r}=\lim_{T\to 0} Z_{g,T, X_r}$ is well defined and 
$$\lim_{r\to \infty}Z_{g,0,X_r}=1.$$
\end{enumerate}
\end{thm}

\begin{equation}\label{------eq:PFtable}
\begin{array}{c|ccc}
\text{Type} & \tilde{A}_r & A_r & B_r,C_r,D_r\\
\hline\\
g=1 & \frac{\theta(q_T)}{\phi(q_T)^2} & \frac{1}{\phi(q_T)^2} & \frac{1}{\phi(q_T)}  \\
\\
g\geq 2 & \theta(q_T) & 1 & 1 
\end{array}
\end{equation}

Point 2 was proved for all types except $\tilde A$ in \cite{Gur}, whereas 1 was  proved in \cite{Lem} for group series $A,\tilde A$, proving independently point 2 for types $\tilde A$. Theorem \ref{__THM:PF} will be proved in 
Section 
\ref{sec:PF}. 
Let us also mention that, whereas the partition function on the plane or a disc in the plane is equal to 1 for any group, the case of the sphere behaves very differently, as the partition function goes to zero exponentially fast, at a speed of order $r^2$ (\cite{MonvS}). 
It further displays a phase transition that was discovered by and named after Douglas and Kazakov \cite{DK} and proved in the case of unitary groups by L\'evy and Ma\"ida \cite{LM}.  See also Boutet de Monvel and Shcherbina \cite{BMS}, where 
the convergence of the right-hand-side below was first proven for a more general class of models. 
\begin{thm}[Douglas--Kazakov phase transition]\label{__THM: DK PT}
For any $T\geq 0$, set $Z_{T,N}$ as the Yang--Mills partition function on the sphere of area $T$ with structure group $\U(N)$. Then the quantity
\[
F(T)=\lim_{N\to\infty}\frac{1}{N^2}\log Z_{T,N}
\]
exists. It defines a function $F\in\mathscr{C}^2(\R_+^*)\cap\mathscr{C}^\infty(\R_+^*\setminus\{\pi^2\})$ which admits a third-order jump at the area $\pi^2$.
\end{thm}

 We will prove Theorem \ref{__THM:PF} in section \ref{sec:Pf PF}. Beforehand, we recall in the next section a well known expression of the partition function using
 representation 
 theory of compact  groups.

% 
% First we can split the limits of partition functions into two categories: the \emph{flat} case, when $g=1$, and the \emph{hyperbolic} case, when $g\geq 2$. Let us justify this 
% terminology: according to the Gauss--Bonnet theorem, the Euler characteristic $2-2g$ of a Riemann surface of genus $g$ is equal to its total curvature. Hence, up to 
% homeomorphism, a Riemann surface of genus $1$ has the same curvature as the torus, which possesses a flat metric, and a Riemann surface of genus $g\geq 2$ has the same 
% curvature as a hyperbolic surface.

\subsection{Character decomposition of the partition function}
\label{sec:Pf Charact}
We shall express the partition functions $Z_{g,T,X_r}$ for any $X\in\{\tilde{A},A,B,C,D\}$ as a sum over non-increasing sequences of integers. The result of this section are well known \cite{Sen0,Lev3}. This will follow from standard representation theory of compact groups, and we will prove it succinctly for the sake of completeness. The main result we want to prove is Prop. \ref{______Prop:fourier_pf}, and the remaining of the section may be skipped by anyone familiar with representation theory. Most of the results we will present can be found in \cite{BtD} and \cite{Far}.

\begin{dfn}
Let $G$ be a compact Lie group.
\begin{enumerate}
\item A \emph{complex representation} of $G$ is a couple $(\varphi,V)$, where $V$ is a complex vector space and $\varphi:G\to\GL(V)$ is a smooth group morphism.
\item A representation $(\varphi,V)$ is \emph{irreducible} if $V$ is the only nontrivial subspace left invariant by $\varphi$.
\item The \emph{dimension} $d_\varphi$ of a representation $(\varphi,V)$ is defined as the dimension of the vector space $V$.
\item The \emph{character} of a representation $(\varphi,V)$ is the function $\chi_\varphi:G\to\C$ defined by
\[
\chi_\varphi(g)=\Tr(\varphi(g)).
\]
\end{enumerate}
\end{dfn}

In the following, we will only consider finite-dimensional representations, unless stated otherwise. The two main results of the representation theory that we will be using are the celebrated \emph{Schur's lemma} and \emph{Plancherel's theorem}.

\begin{lem}[Schur's lemma]
\begin{enumerate}
\item Let $(\varphi_1,V_1)$ and $(\varphi_2,V_2)$ be two irreducible representations of $G$. If $A:V_1\to V_2$ is a nonzero linear map such that
\begin{equation}\label{------eq:Schur1}
A\varphi_1(g)=\varphi_2(g)A,\ \forall g\in G,
\end{equation}
then $A$ is an isomorphism.
\item Let $(\varphi,V)$ be an irreducible representation of $G$. If $A\in\End(V)$ satisfies
\begin{equation}\label{------eq:Schur2}
A\varphi(g)=\varphi(g)A,\ \forall g\in G,
\end{equation}
then there exists $\a\in\C$ such that $A=\a I$, where $I$ denotes the identity of $V.$
\end{enumerate}
\end{lem}

Among the numerous consequences of this lemma, one can state that the relation \eqref{------eq:Schur1} defines an equivalence relation between irreducible representations of $G$, and it enables to define the set $\widehat{G}$ of equivalence classes of irreducible representations. This set, sometimes called \emph{dual space} of $G$, is countable whenever $G$ is compact, and it is also a group when $G$ is abelian. Another consequence of Schur's lemma is that two irreducible representations within the same class $\lambda\in\widehat{G}$ have same dimension $d_\lambda$ and same character $\chi_\lambda$.

\begin{thm}[Plancherel's theorem]
For any continuous function $f\in C^\infty(G)$, the following sum absolutely converges
\begin{equation}
f(g)=\sum_{\lambda\in\widehat{G}}d_\lambda (f*\chi_\lambda)(g),   \forall g\in G,
\end{equation}
where for any $\chi\in C(G),$ $f*\chi(h)=\int_G f(h)\chi(h^{-1}g )dh,\forall g\in G. $
\end{thm}

From these results one can prove the following, which is a particular instance of Thm. 4.2 in \cite{Lia}.

\begin{thm}
Let $(p_t)_{t>0 }$ be the heat kernel on $G$.  Then for all $t>0,$ $p_t\in C^\infty(G) $ and the following sum
\begin{equation}\label{------eq:Fourier_HK}
p_t(g)=\sum_{\lambda\in\widehat{G}}e^{-\frac{t}{2}c_\lambda}d_\lambda \chi_\lambda(g),\ \forall g\in G,
\end{equation}
 absolutely converges, where $c_\lambda\geq 0$ is the non-negative real number such that
\[
\Delta_G\chi_\lambda=-c_\lambda \chi_\lambda.
\]
\end{thm}

Before we state the Fourier decomposition of Yang--Mills partition function, let us also mention an easy but useful result about the irreducible characters of compact groups.

\begin{prop}
Let $\lambda\in\widehat{G}$ be an equivalence class of irreducible representations of a compact group $G$. We have
\begin{equation}
\int_{G} \chi_\l(x [y,z])dy=d_\l^{-1} \chi_\l(xz^{-1})\chi_\l(z),\forall x,z\in G\label{------eq:int_conj}
\end{equation}
and 
\begin{equation}\label{------eq:int_commu}
\int_{G^2} \chi_\lambda(x[y,z])dydz =\frac{\chi_\lambda(x)}{d_\lambda^2},\ \forall x\in G. 
\end{equation}
\end{prop}

\begin{proof}
Consider   an element $(\varphi,V)$ of the equivalence class $\l$ and set  
$$A_z=\int_{G}\varphi(yzy^{-1})dy,\forall z\in G.  $$ 
Since the Haar measure is invariant
by multiplication,   $A_z$ satisfies \eqref{------eq:Schur2} of Schur's lemma. Therefore,
$A_z= \a_z I$   for some scalars $\a_z \in \C$ with 
$$\a_z  d_\l= \Tr(A_z)= \int_{G} \chi_\l (y z y^{-1})dy= \chi_\l(z).$$
We can now write the left-hand side of \eqref{------eq:int_conj} as
$$\Tr( \varphi(x) A_z \varphi(z^{-1}))= \frac{\chi_\l(z)}{d_\l} \Tr(\varphi(z^{-1} x ))= d_\l^{-1}\chi_\l(z) \chi_\l(z^{-1}x), \forall x,z\in G,$$
while using Plancherel theorem, the left-hand side of \eqref{------eq:int_commu} reads
$$d_\l^{-1}\int_{G}  \chi_\l(z) \chi_\l(z^{-1}x) dz=\frac{1}{d_\l}\chi_\l*\chi_\l(x)= \frac{\chi_\l(x)}{d_\l^2},\forall x\in G.$$

%Consider   an element $(\varphi,V)$ of the equivalence class $\l$ and set  $$A_z=\int_{G}\varphi(yzy^{-1})dy,\forall z\in G \text{ and } B= \frac{1}{d_\l}\int_G   \chi_\l(h) \varphi(h^{-1})dh.  $$ 
%Then  by invariance of the Haar measure 
%by multiplication and of $\chi_\l$ by conjugation, $A_z$ and $B$ satisfy \eqref{------eq:Schur2} of Schur's lemma. Therefore,
%$A_z= \l I$  and $B= \mu I$  for some scalars $\l_z,\mu \in \C$ with 
%$$\l  d_\l= \Tr(A_z)= \int_{G} \chi_\l (y z y^{-1})dy= \chi_\l(z) \text{ and } \mu d_\l=\frac{1}{d_\l} \chi_\l*\chi_\l(1)= \frac{\chi_\l(1)}{d_\l^2}= \frac{1}{d_\l}, $$
%where we used the Plancherel theorem in the last inequality.  We can now write the left-hand-side of \eqref{------eq:int_commu} as 
%$$\int_{G}\Tr( \varphi(x) A_z z^{-1}) dz=\int_{G}  \frac{\chi_\l(z)}{d_\l}\Tr(\varphi(x z^{-1}))dz=\Tr(\varphi(x)B)$$
%where $ B= \frac{1}{d_\l}\int_G   \chi_\l(z) \varphi(z^{-1})dz.  $ For all $g\in G,$
%$$\varphi(g)$$
\end{proof}

\begin{prop}\label{______Prop:fourier_pf}
Let $\Sigma$ be a connected orientable surface of genus $g\geq 1$. If $\partial\Sigma$ is connected, then the Yang--Mills partition function on $\Sigma$ with structure group of type $X_r$ and boundary condition $t\in G/\mathrm{Ad}$ is given by
\begin{equation}\label{------eq:fourier_pf_b}
Z_{(g,1),T,X_r}(t)=\sum_{\lambda\in\widehat{G}} e^{-\frac{T}{2}c_\lambda}d_\lambda^{1-2g}\chi_\lambda(t).
\end{equation}
If $\Sigma$ has no boundary, then the Yang--Mills partition function on $\Sigma$ with structure group $G$ is given by
\begin{equation}\label{------eq:fourier_pf}
Z_{g,T,X_r}=\sum_{\lambda\in\widehat{G}} e^{-\frac{T}{2}c_\lambda}d_\lambda^{2-2g}.
\end{equation}
\end{prop}

\begin{proof}
Let us start with the simplest case, which is \eqref{------eq:fourier_pf}. Consider the area-weighted map $(M,a)$ of genus $g$ with $1$ vertex,   $2g$ edges $(a_1,b_1,\ldots,a_g,b_g)$ and $1$ face $f$ with area $T>0$ and boundary  
$$\pl f=  b_g^{-1}a_{g}^{-1}b_ga_g\cdots b_1^{-1}a_1^{-1}b_1a_1. $$ 
%generating the fundamental group $\pi_1(\Sigma)$, and let $\mathbb{G}=(\mathbb{V},\mathbb{E},\mathbb{F})$ be the graph on $\Sigma$ composed by 1 face $F$ of area $T$, whose boundary is given by $\partial F=[a_1,b_1]\cdots[a_g,b_g]$. 
Then
\[
Z_{g,T,X_r}=Z_{M,a,G}=\int_{G^{2g}} p_T([x_1,y_1]\cdots[x_g,y_g])dx_1dy_1\cdots dx_gdy_g. 
\]
Using the Fourier decomposition of the heat kernel \eqref{------eq:Fourier_HK}, we get
\[
Z_{g,T,X_r}=\sum_{\lambda\in\widehat{G}(N)}e^{-\frac{T}{2}c_\lambda}d_\lambda\int_{G^{2g}}\chi_\lambda([x_1,y_1]\cdots[x_g,y_g])dx_1dy_1\cdots dx_gdy_g.
\]
We now integrate out all commutators using \eqref{------eq:int_commu}, which yields\footnote{The following formula goes back to Frobenius for finite groups, see e.g. \cite[Sect. 7.9]{Sen7}.}
\[
\int_{G^{2g}}\chi_\lambda([x_1,y_1]\cdots[x_g,y_g])dx_1dy_1\cdots dx_gdy_g=\frac{1}{d_\lambda^{2g-1}}.
\]
The results follows. The proof of \eqref{------eq:fourier_pf_b} is similar, using an area-weighted map with one face whose boundary is $\partial f=b_g^{-1}a_{g}^{-1}b_ga_g\cdots b_1^{-1}a_1^{-1}b_1a_1\ell$, where $\ell$ is the simple loop corresponding to $\partial\Sigma$, oriented positively, for any $h\in G$ to 
\begin{equation}
Z_{(g,1),T,X_r}(h)=Z_{M,a,G}(h)=\int_{G^{2g}} p_T(h[x_1,y_1]\cdots[x_g,y_g])dx_1dy_1\cdots dx_gdy_g.  \label{-----eq:PF one Boundary}
\end{equation}
\end{proof}

Now we will specify \eqref{------eq:fourier_pf} to all compact classical groups: to do this, we only need to know the dual space $\widehat{G}$, and the numbers $c_\lambda$ and $d_\lambda$ for every $\lambda\in\widehat{G}$. This can be done thanks to their root systems. Most definitions and results are borrowed from \cite{BtD}, and can also be recovered with much clarity from Sections 2.2 and 2.3 in \cite{Mel}.

\begin{dfn}
Let $G$ be a compact connected Lie group.
\begin{enumerate}
\item A \emph{weight} of a representation $(\pi,V)$ of $G$ is a group morphism $\omega:T\to\U(1)$ where $T$ is a maximal torus of $G$, such that the space $V^\omega=\{v\in 
V: \rho(t)v=\omega(t)v,\ \forall t\in T\}$ is not trivial. The weights form a lattice $I_r=\Z\Omega=\bigoplus_{i=1}^r\Z\omega_i$, where $(\omega_i)$ is a distinguished basis of the 
lattice and $r$ is the rank of the weight lattice. The normaliser of $T$ in $G$ acts by conjugation on $T$ yielding an action on $I_r.$  The vector space $\mathscr{V}
_r=\R\Omega=I_r\otimes_\Z\R$ 
can be endowed with an invariant 
inner product $\langle\cdot,\cdot\rangle$. In fact, if $\mathfrak{t}$ is the Lie algebra of the maximal torus $T\subset G$, we have $\R\Omega\simeq\mathfrak{t}^*$ and $\langle\cdot,\cdot\rangle$ can be taken as the dual of the inner product \eqref{------eq:inner P LieAlg}.
\item A \emph{root} of $G$ is a non-zero weight of the adjoint representation of $G$. The \emph{root system} $\Phi$ can be split into $\Phi_+\sqcup\Phi_-$ with $\Phi_-=-\Phi_+$. 
 The elements of $\Phi_+$ are called \emph{positive roots}. 
\item The \emph{Weyl chamber} $\mathscr{C}_r$ is the set
\[
\mathscr{C}_r=\{x\in\mathscr{V}_r:\langle x,\alpha\rangle> 0\ \forall \alpha\in\Phi_+\}.
\]
It is an open cone in $\mathscr{V}_r$.
\item A \emph{dominant weight} is a weight that belongs to the closure $\overline{\mathscr{C}_r}$ of the Weyl chamber . We denote by $\Lambda_r$ the set of dominant weights. 
\item The element 
$$\rho=\frac{1}{2}\sum_{\a \in \Phi_+} \a$$
has the property that for any $\omega\in I_r$
\begin{equation}
 \omega\in \overline{\mathscr{C}_r} \text{ if and only if } \rho+ \omega \in \mathscr{C}_r.\label{------eq:smallest element}
\end{equation}
%\begin{equation}
%\rho+ ( I_r\cap \overline{\mathscr{C}_r})=(\rho+I_r)\cap \mathscr{C}_r. \label{------eq:smallest element}
%\end{equation}
\end{enumerate}
\end{dfn}

Below are listed the root systems corresponding to the classical groups depending on their root systems. Note that we do not treat the root systems of exceptional Lie algebras $E_6,E_7,E_8$, $F_4$ and $G_2$, as we focus in this paper on unitary, special unitary, orthogonal and symplectic groups.

\[
\begin{array}{c|c|c}
\text{Group} & \text{Type} & \Phi \\
\hline
\U(r) & \tilde{A}_r & \{e_i-e_j,1\leq i,j\leq r, i\neq j\} \\
\SU(r+1) & A_r & \{e_i-e_j,1\leq i, j\leq r+1, i\neq j\} \\
\SO(2r+1) & B_r & \{\pm e_i\pm e_j, 1\leq i<j\leq r\}\cup\{\pm e_i,1\leq i\leq N\} \\
\Sp(r) & C_r & \{\pm e_i\pm e_j, 1\leq i<j\leq r\}\cup\{\pm 2e_i,1\leq i\leq N\} \\
\SO(2r) & D_r & \{\pm e_i\pm e_j, 1\leq i<j\leq r\} \\
\end{array}
\]
For $\SU(r+1)$, we identify $\mathscr{V}
_r$ with $\mathscr{V}_r=\R^{r+1}/\R(1,\ldots,1)$ endowed with the norm
\[
\Vert [x]\Vert^2=\frac{1}{N+1}\sum_{i=1}^{N+1}\big(x_i-\frac{1}{N+1}\sum_{j=1}^{N+1} x_j\big)^2,\ \forall [x]\in\R^{N+1}/\R(1,\ldots,1).
\]
For any classical group $G$ associated with the type $X_r$, with $X\in\{\tilde{A},A,B,C,D\}$, there is a bijection $\widehat{G}\simeq\Lambda_r$: this is a consequence of the highest-weight theory (see \cite{BtD} for instance). From now on, $\lambda$ will indistinctly represent an equivalence class of irreducible representations or an element of $
\Lambda_r$. 

\begin{prop}
Let $G_N$ be a classical group associated with 
one of the types $\tilde{A}_r,A_r,B_r,C_r,D_r$ and $\lambda\in\widehat{G}_N\simeq\Lambda_r$ be an equivalence class of irreducible representations. We have
\begin{equation}\label{------eq:cas}
c_\lambda=\langle\lambda+2\rho,\lambda\rangle
\end{equation}
and
\begin{equation}\label{------eq:dim}
d_\lambda=\frac{\prod_{\alpha\in\Phi_+}\langle\lambda+\rho,\alpha\rangle}{\prod_{\alpha\in\Phi_+}\langle\rho,\alpha\rangle}.
\end{equation}
\end{prop}
We list below some of the objects defined previously, specified for the corresponding classical groups. Most of the explicit computations can be found in \cite{BtD}.
\begin{itemize}
\item $\U(r)$: type $\tilde{A}_r$
\begin{align*}
& \Phi_+ = \{e_i-e_j,\ 1\leq i<j\leq r\}\\
& \mathscr{C}_r = \{x\in\R^r:x_1\geq \cdots\geq x_r\}\\
& \Lambda_r = \{\lambda\in\Z^r:\lambda_1\geq\cdots\geq\lambda_r\}\\
& \rho = \big(\frac{r-1}{2},\frac{r-3}{2},\ldots,\frac{3-r}{2},\frac{1-r}{2}\big)\\
& c_\lambda = \frac{1}{r}\sum_{i=1}^r \lambda_i(\lambda_i+r+1-2i),\ \forall \lambda\in\Lambda_r\\
& d_\lambda = \prod_{1\leq i<j\leq r}\frac{\lambda_i-\lambda_j+j-i}{j-i},\ \forall \lambda\in\Lambda_r
\end{align*}
\item $\SU(r+1)$: type $A_r$
\begin{align*}
& \Phi_+ = \{[e_i-e_j], 1\leq i<j\leq r+1\} \\
& \mathscr{C}_r = \{[x]\in\R^{r+1}/\sim: x_1\geq \cdots\geq x_{r+1}\} \\
& \Lambda_r = \{[\lambda]\in\Z^{r+1}/\R(1,\ldots,1):\lambda_1\geq\cdots\geq\lambda_{r+1}\}\\
& \Lambda_r \simeq\{\lambda\in\Z^{r+1}:\lambda_1\geq\cdots\geq\lambda_{r+1}=0\}\\
& \rho = \big[\big(\frac{r}{2},\frac{r}{2}-1,\ldots,1-\frac{r}{2},-\frac{r}{2}\big)\big] \\
& c_\lambda = \frac{1}{r+1}\sum_{i=1}^{r+1} \lambda_i(\lambda_i+r+2-2i)-\frac{1}{(r+1)^2}\left(\sum_{i=1}^{r+1}\lambda_i\right)^2,\ \forall[\lambda]\in\Lambda_r \\
& d_\lambda = \prod_{1\leq i<j\leq r+1}\frac{\lambda_i-\lambda_j+j-i}{j-i},\ \forall[\lambda]\in\Lambda_r \\
\end{align*}

\item $\SO(2r+1)$: type $B_r$
\begin{align*}
& \Phi_+=\{e_i\pm e_j, 1\leq i<j\leq r\}\cup\{e_i,1\leq i\leq r\}  \\
& \mathscr{C}_r=\{x\in\R^r: x_1\geq \cdots\geq x_r\geq 0\} \\
& \Lambda_r=\{\lambda\in\N^r:\lambda_1\geq\cdots\geq\lambda_r\}\\
& \rho = \big(r-\frac{1}{2},r-\frac{3}{2},\ldots,\frac{1}{2}\big) \\
& c_\lambda = \frac{1}{2r+1}\sum_{i=1}^r\lambda_i(\lambda_i+2r+1-2i),\ \forall\lambda\in\Lambda_r \\
& d_\lambda =\prod_{1\leq i<j\leq r}\frac{\lambda_i-\lambda_j+j-i}{j-i}\prod_{1\leq i\leq j\leq r}\frac{\lambda_i+\lambda_j+2r+1-i-j}{2r+1-i-j},\ \forall\lambda\in\Lambda_r
\end{align*}
\item $\Sp(r)$: type $C_r$
\begin{align*}
& \Phi_+=\{e_i\pm e_j,1\leq i<j\leq r\}\cup\{2e_i,1\leq i\leq r\} \\
& \mathscr{C}_r=\{x\in\R^r: x_1\geq \cdots\geq x_r\geq 0\} \\
& \Lambda_r=\{\lambda\in\N^r:\lambda_1\geq\cdots\geq\lambda_r\}\\
& \rho = (r,r-1,\ldots,1)\\
& c_\lambda = \frac{1}{2r}\sum_{i=1}^r \lambda_i(\lambda_i+2r+2-2i),\ \forall\lambda\in\Lambda_r \\
& d_\lambda = \prod_{1\leq i<j\leq r}\frac{\lambda_i-\lambda_j+j-i}{j-i}\prod_{1\leq i\leq j\leq r}\frac{\lambda_i+\lambda_j+2r+2-i-j}{2r+2-i-j},\ \forall\lambda\in\Lambda_r
\end{align*}
\item $\SO(2r)$: type $D_r$
\begin{align*}
& \Phi_+=\{e_i\pm e_j,1\leq i<j\leq r\}\\
& \mathscr{C}_r=\{x\in\R^r:x_1\geq x_2\geq\cdots\geq|x_r|\} \\
& \Lambda_r=\{\lambda\in\Z^r:\lambda_1\geq\cdots\geq\lambda_{r-1}\geq|\lambda_r|\}\\
& \rho = (r-1,r-2,\ldots,0) \\
& c_\lambda = \frac{1}{2r}\sum_{i=1}^r \lambda_i(\lambda_i+2r-2i),\ \forall\lambda\in\Lambda_r \\
& d_\lambda = \prod_{1\leq i<j\leq r}\frac{(\lambda_i-\lambda_j+j-i)(\lambda_i+\lambda_j+2r-i-j)}{(j-i)(2r-i-j)},\ \forall\lambda\in\Lambda_r
\end{align*}
\end{itemize}

We will prove Theorem \ref{__THM:PF} using an asymptotic estimation of \eqref{------eq:fourier_pf}.
\subsection{Proof of convergence of partition functions}
\label{sec:Pf PF}
We first recall the result of \cite{Gur} and give an alternative proof based on the result of \cite{Lem} for $A_r$ series.

\subsubsection{Witten zeta function}

When $g\ge 2$, for the root systems $A_r,B_r,C_r$ and $D_r$, a proof relies on an asymptotic estimation of the \emph{Witten zeta function}\footnote{The following function was introduced in \cite{Wit} and seems to have been named after Witten in \cite{Zagier}.} 
\begin{equation}
\zeta_{X_r}(s)=\sum_{\lambda\in\Lambda_r}\frac{1}{d_\lambda^s},\ \forall s>1,\ \forall X\in\{A,B,C,D\}.
\end{equation}

The first claim we can do is that, for any root system $X_r$ with $X\in\{A,B,C,D\}$, any $T\in\R_+$ and any $g\geq 2$, we have
\begin{equation}\label{------eq:encadre_zeta}
1\leq \sum_{\lambda\in\Lambda_r} e^{-\frac{T}{2}c_\lambda}d_\lambda^{2-2g} \leq\zeta_{X_r}(2).
\end{equation}
Indeed, the sum in the middle contains only nonnegative terms, therefore is bounded from below by the term corresponding to $\lambda=(0,\ldots,0)\in\Lambda_r$, which is equal to 1. The upper bound comes from the fact that for any $\lambda\in\Lambda_r$ and $g\geq 2$, the number $c_\lambda$ is nonnegative, and $d_\lambda^{2g-2}\geq d_\lambda^2$. Thanks to \eqref{------eq:encadre_zeta}  we only need to prove
\begin{equation}\label{------eq:lim_zeta}
\lim_{r\to\infty} \zeta_{X_r}(2)=1
\end{equation}
for $X\in\{A,B,C,D\}$ in order to prove Theorem \ref{__THM:PF} for $g\geq 2$.

\begin{prop}[\cite{Gur}]\label{______Prop:lim_zeta}
For any real $s>1$ and $X\in\{A,B,C,D\}$, one has
\[
\lim_{r\to\infty}\zeta_{X_r}(s)=1.
\]
\end{prop}

\begin{proof}
The case of $A_r$ is detailed in \cite{Lem}, we argue here that it implies the other cases. We will compare the dimensions and the sets of dominant weights for different root 
systems, thus we will denote by $\Lambda_{X_r}$ the set of dominant weights of type $X_r$ and $d_{X_r,\lambda}$ the dimension of an irreducible representation of the 
compact connected group of type $X_r$ with highest weight $\lambda$. First, let us remark that the sets of dominant weights are the same for $A_r,B_r$ and $C_r$. Let 
$X\in\{B,C\}$ and $\lambda\in\Lambda_{X_r}$ a dominant weight different from $(0,\ldots,0)$. As we have $\lambda_i\geq 0$ for any $i$, it is clear that
\[
d_{X_r,\lambda}\geq \prod_{1\leq i<j\leq r}\frac{\lambda_i-\lambda_j+j-i}{j-i}=d_{A_r,\lambda}.
\]
It follows that, for any $s>1$:
\[
0\leq \sum_{\substack{\lambda\in\Lambda_{X_r}\\ \lambda\neq(0,\ldots,0)}}\frac{1}{(d_{X_r,\lambda})^s}\leq \sum_{\substack{\lambda\in\Lambda_{A_r}\\ \lambda\neq(0,\ldots,0)}}\frac{1}{(d_{A_r,\lambda})^s}.
\]
We get the right limit by letting $r$ go to infinity.

In the case of $D_r$, some dominant weights have a negative coefficient, therefore an extra care must be taken. Let $\lambda\in\Lambda_{D_r}$ and $1\leq i<j\leq r-1$, we have $\lambda_i+\lambda_j\geq 0$ as the sum of nonnegative integers. Now, if $1\leq i\leq r-1$, we also know that $\lambda_i\leq\lambda_{r-1}\leq|\lambda_r|$ therefore $\lambda_i-\lambda_r\geq 0$ as well. In any case, we deduce that
\[
d_{D_r,\lambda}\geq \prod_{1\leq i<j\leq r}\frac{\lambda_i-\lambda_j+j-i}{j-i}=d_{A_r,\tilde{\lambda}},
\]
where $\tilde{\lambda}=(\lambda_1,\ldots,\lambda_{r-1},|\lambda_r|)$. Any dominant weight $\lambda$ of type $A_r$ corresponds to at most two different dominant weights of type $r_N$: the same weight, and the one obtained by changing the sign of $\lambda_r$. Hence, we have
\[
0\leq\sum_{\substack{\lambda\in\Lambda_{D_r}\\ \lambda\neq (0,\ldots,0)}} \frac{1}{(d_{D_r,\lambda})^s} \leq 2\sum_{\substack{\lambda\in\Lambda_{A_r}\\ \lambda\neq(0,\ldots,0)}} \frac{1}{(d_{A_r,\lambda})^s}.
\]
The right-hand side converges to 0 when $r\to\infty$ therefore we obtain the expected limit for $D_r$.
\end{proof}

\begin{proof}[Proof of Theorem \ref{__THM:PF} with $g\geq 2$]
According to Prop. \ref{______Prop:lim_zeta}, for any $g\geq 2$ and  $X\in\{A,B,C,D\}$ we have
\[
\lim_{r\to\infty}\zeta_{X_r}(2g-2)=1.
\]
Together with \eqref{------eq:encadre_zeta}, we get 
\[
\lim_{r\to\infty} Z_{g,T,X_r}=1
\]

\end{proof}

\subsubsection{Discrete Gaussian random variables in a cone}

In this section we prove that the convergence of $Z_{1,T,X_r}$ as $r\to \infty$ and $T>0$ is fixed for types $\{B,C,D\}$.   
\begin{dfn}
Let $r\geq 1$ be an integer and $\Lambda_r$ be the set of dominant elements of type $X_r$, with $X\in\{\tilde{A},A,B,C,D\}$. A random variable $\mu$ on $\Lambda_r$ is 
\emph{Gaussian} with parameter $t$ if for any $\Lambda\subset\Lambda_r$
\begin{equation}\label{------eq:disc_gauss}
\P(\mu\subset\Lambda)=\frac{1}{Z_{r,t}}\sum_{\lambda\in\Lambda}e^{-\frac{t}{2}c_\lambda}
\end{equation}
and the associated \emph{partition function} is
\begin{equation}\label{------eq:pf_gauss_disc}
Z_{r,t}=\sum_{\lambda\in\Lambda_r}e^{-\frac{t}{2}c_\lambda}.
\end{equation}
\end{dfn}

The denomination `Gaussian'  is justified by the following remark. Recall that for any classical group $G_N\subset\GL_n(\C)$ of type $X_r$ and any $\lambda\in\widehat{G}_N$,
\[
c_\lambda=\frac{1}{n}\langle\lambda+\rho,\lambda\rangle=\frac{1}{n}(\Vert\lambda+\rho\Vert^2-\Vert\rho\Vert^2).
\]
It follows that for any $\lambda\in\Lambda_r,$
\[
\P(\mu=\lambda)\propto e^{-\frac{t}{2n}\Vert\lambda+\rho\Vert^2},
\]
which  thanks to \eqref{------eq:smallest element} defines, after a shift by $\rho$, a discrete Gaussian distribution conditioned to belong to $\rho+(I_r\cap \overline{\mathscr{C}_r})=(\rho+I_r)\cap \mathscr{C}_r$. Note that if $G_N$ is a 
classical group of type $X_r$, the partition function $Z_{r,T}$ of the Gaussian distribution on $\Lambda_r$ with parameter $T$ is equal to the Yang--Mills partition function $Z_{1,T,X_r}$ on a torus of area $T$ with structure group $G_N$. Before we prove Theorem \ref{__THM:PF} for $g=1$, let us give a few notations that we will use. When $
\lambda=(\lambda_1,\lambda_2,\ldots)$ is a sequence of non-negative real numbers, we denote by $|\lambda|=\sum_{i\geq 1} \lambda_i$ its total sum. We also denote by
\[
\mathcal{P}=\{(\alpha_1,\alpha_2,\ldots)\in\N^{\N^*}:\alpha_1\geq\alpha_2\geq\cdots\}
\]
the set of integer partitions. For any $\alpha\in\mathcal{P}$ we write $\ell(\alpha)$ for its number of non-zero parts. It is well known that the generating function of partitions is given by the inverse of Euler function:
\begin{equation}
\sum_{\alpha\in\mathcal{P}}q^{|\alpha|}=\frac{1}{\phi(q)},\ \forall q\in\C \ \mathrm{s.t.} \ |q|<1.
\end{equation}

\begin{proof}[Proof of Thm \ref{__THM:PF} with $g=1$]
\emph{Types B and C}. Let us define an increasing sequence  setting $\check{\rho}=(r-\rho_1,\ldots,r-\rho_r)$ and write
\[
\langle\lambda+2\rho,\lambda\rangle=\Vert\lambda\Vert^2+2\sum_{i=1}^r\lambda_i\rho_i=2r|\lambda|+\Vert\lambda\Vert^2-2\sum_{i=1}^{\ell(\lambda)}\lambda_i\check{\rho}_i.
\]
For the types $B_r$ and $C_r$ we have $\Lambda_r\simeq\{\lambda\in\mathcal{P}:\ell(\lambda)\leq r\}$ therefore we have
\[
Z_{r,T}=\sum_{\lambda\in\mathcal{P}:\ell(\lambda)\leq r} e^{-\frac{rT}{n}|\lambda|} e^{-\frac{T}{2n}(\Vert\lambda\Vert^2-2\sum_{i=1}^{\ell(\lambda)}\lambda_i\check{\rho}_i)}.
\]
%In both cases we have
%\[
%0\leq\sum_{i=1}^{\ell(\lambda)}\check{\rho}_i\leq \frac{\ell(\lambda)^2}{2}.
%\]
%Indeed, the first inequality follows from the fact that $\rho_i\geq0$ for any $i$, and the second one from the fact that {\color{red} TO DETAIL}. 
Since $\check{\rho}$ is increasing and $\lambda$ is non-increasing,
%while both are non-negative
\[
\sum_{1\leq i,j\leq\ell(\lambda)} (\lambda_i-\lambda_j)(\check{\rho}_i-\check{\rho}_j)\leq 0.
\]
It follows that
\[
0\leq 2\sum_{i=1}^{\ell(\lambda)} \lambda_i\check{\rho}_i\leq \frac{2}{\ell(\lambda)}\sum_{i=1}^{\ell(\lambda)} \lambda_i\sum_{i=1}^{\ell(\lambda)}\check{\rho}_i\leq |\lambda|
\ell(\lambda),
\]
where we used for the last inequality that   in $B_r$ and $C_r$ cases, $\sum_{i=1}^{\ell(\l)}\check{\rho}_i$ is either $\frac{\ell(\l)^2}{2}$ or $\frac{\ell(\l)(\ell(\l)-1)}{2}.$ We obtain
\[
\sum_{\lambda\in\mathcal{P}:\ell(\lambda)\leq r} e^{-\frac{r T}{n}|\lambda|-\frac{T}{2n}\Vert\lambda\Vert^2}\leq Z_{r,T}\leq \sum_{\lambda\in\mathcal{P}:\ell(\lambda)\leq r} e^{-(\frac{rT}{n}-\frac{\ell(\lambda)T}{2n})|\lambda|-\frac{T}{2n}\Vert\lambda\Vert^2}.
\]
Recalling that $n$ is $2r+1$ or $2r$, both sides converge to $\phi(q_T)^{-1}$ when $r\to\infty$ by dominated convergence, and we therefore obtain the expected limit.

\emph{Type $D$.} In this case, let us introduce two notations. For $r\ge2, \lambda=(\lambda_1,\ldots,\lambda_r)$ such that $\lambda_1\geq\cdots\geq\lambda_r$ and $m\in\Z$, we set $\tilde{\lambda}=(\lambda_1,\ldots,\lambda_{r-1})$ and $\lambda+m=(\lambda_1+m,\ldots,\lambda_r+m)$. Then
\[
\langle\lambda+2\rho,\rho\rangle=\Vert\lambda\Vert^2+2\sum_{i=1}^{r-1}\lambda_i(r-i)=2r|\tilde{\lambda}|+\lambda_r^2+\Vert\tilde{\lambda}\Vert^2-2\sum_{i=1}^{r-1}i\tilde{\lambda}_i
\]
and
\begin{align*}
Z_{r,T} = & \sum_{\substack{\lambda\in\mathcal{P},k\in\Z\\ \ell(\lambda)\leq r-1,|k|\leq \lambda_{r-1}}} e^{-\frac{rT}{n}|\lambda|-\frac{T}{2n}(k^2+\Vert\lambda\Vert^2-2\sum_{i=1}^{r-1} i\lambda_i)}\\
= &  \sum_{\substack{\lambda\in\mathcal{P} \\ \ell(\lambda)\leq r-1 }}e^{-\frac{rT}{n}|\lambda|-\frac{T}{2n}( \Vert\lambda\Vert^2-2\sum_{i=1}^{r-1} i\lambda_i)}\left(\sum_{|k|\leq \l_{r-1}} e^{-\frac{Tk^2}{2n}}\right).
\end{align*}
Using the same argument as for the types $B$ and $C$, we can bound the summand  of the sum over $\mathcal{P}$ by 
\[q_{\frac T2}^{|\l|} (2\l_1+1)\]
and by dominated convergence, 
\[\lim_{r\to\infty}Z_{r,T}=   \sum_{ \lambda\in\mathcal{P} } q_T^{|\l|}= \frac{1}{\phi(q_T)}.\]

\end{proof}

\begin{rmk} For any $T>0,$ it is also possible to deduce the case $g\ge 2$ from the case $g=1$ using that when $G$ is of type $X_r$, for any non-trivial representation $\l\in \widehat{G},$ $d_\l\ge r.$
\end{rmk}
\begin{rmk} A proof for the groups  $\U(N)$ and $\SU(N+1)$ using heighest weights can be produced similarly to the other cases yielding a  different proof of the result of 
\cite{Lem}, the latter being  formulated in terms of Young diagrams.
\end{rmk}

\section{Proof of Wilson loops convergence}\label{sec:}
In this section, we will prove theorems \ref{__THM:Disc dYM} and \ref{__THM:simple non-contrac}.  The proof relies on an absolute continuity relation and the boundedness of the Yang--Mills partition
 function,\footnote{which follows from the convergence discussed in the Section \ref{sec:PF}.}  except for a class of loops considered in \ref{sec:Non-Sep Loop}, which requires an independent argument.  Theorem
 \ref{__THM:Disc dYM} and its proof can be extended to the continuous setting of \cite{Lev2} leading to Theorem \ref{__THM: Disc YMCont}.  The needed additional arguments being very close to \cite[Sect. 3]{Lev2} and \cite[Sect. 4]{CDG}  will only be sketched 
 at the end of the first following section. From here onwards, $G_N$ will denote a classical group of size $N$.  

\subsection{Loops within in a disc}
\label{sec: Loops in a disc}
Assume that $(M,a)$ is an area-weighted map on $\Sigma$  of genus $g\ge 1$ and area $|a|=T$, that $U$ is a topological closed disc of $\Sigma$ obtained by union of faces of $M$, with area $|a_U|$, such that $u=T-|a_U|>0$. The boundary of $U$ is given by a loop $\pl U\in\mathrm{L}(\tilde{M}_U).$ 
\begin{lem} \label{_____Lem:Density dYM}The measure  $\big(\mathscr{R}_{\tilde{M}_U}^{M}\big)_*(\YM_{M,a,G_N})$ has density
\begin{equation*}
\frac{Z_{(g,1),u,G_N} (h_{\partial U}^{-1})}{Z_{g,T,G_N}}
\end{equation*}
 with respect to $\YM_{\tilde M_U,a_U,G_N}.$
\end{lem}
\begin{proof} Thanks to Lemma \ref{_____Lem: Compat}, we can assume that $M$ has same number of faces as $\tilde M_U,$ that is, that exactly one face of $M$ is not included  in $U$. Denote by $k$ be the number of edges of $\pl U.$ Let $M'$ be the map with $k+1$  vertices, $k+2g+ 1$ edges 
$u_1,\ldots, u_k,e,a_1,b_1,\ldots,a_g,b_g$ and two faces $U'$ and $V_g$, such that $\pl U'=u_k\cdots u_1$ and  $\pl V_g=e^{-1} [b_g^{-1},a_g^{-1}]\cdots [b_a^{-1},a_1^{-1}]eu_1^{-1} \cdots u_k^{-1}$. Let us assume that $M$ is 
obtained by 
gluing $M'$ with $\tilde M_U$,  
identifying $\pl U$ with $\pl U'$ so that the base of $\pl U$ is sent to an endpoint of $e$.  For any $h \in \mathcal{M}(\mathrm{P}(M),G_N)$, let us write $x_l=h_{a_l},y_l= h_{b_l}$ for all $1\le l\le g$ and $z=h_{e}.$  Denote the set of non-marked faces of $\tilde M_U$  by $\Fbb_U$. Then $\YM_{M,a,G_N}(dh)$ can be written as
\begin{align*}
\frac 1{Z_{g,T,G_N}}p_{u}(h_{\pl U}^{-1}z[x_1,y_1]\cdots [x_g,y_g] z^{-1} )dz \prod_{i=1}^gdx_i dy_i\prod_{f\in \Fbb_U}p_{a_f}(h_{\pl f}) \mathrm{U}_{\tilde{M}_U,G_N}(dh).
\end{align*}
Integrating over $(x_i,y_i)_i$  and $z$ the result follows by invariance by conjugation of the Haar measure and \eqref{-----eq:PF one Boundary}.
%It follows that under $\YM_{M,a,G_N},$ the law of $\mathscr{R}_{\tilde{M}_U}^M$ is absolutely continuous with respect to $\mathrm{U}_{M,G_N}$ with density
%\[
%\frac {Z_{g,u,G_N}([h_{\pl U}^{-1}])} {Z_{g,T,X_r}}, \ \forall h\in \mathcal{M}(\mathrm{P}(M),G_N),
%\]
%which is the expected result.
\end{proof}

\begin{proof}[Proof of Thm. \ref{__THM:Disc dYM}] Without loss of generality, up to refinement of $M,$ we can assume that $U$ is a topological closed disc, that is the union of faces of $M.$ Let us first assume that $|a_U|<T=|a|.$ For any $h\in G_N$ and $\l\in \widehat{G}_N,$ since the representation is unitary, $\|\chi_\l\|_\infty=\chi_\l(1)= d_\l$. For any $s>0$, 
%if we set $\psi_{g,u,G_N}:G_N\to\C,g\mapsto Z_{g,u,G_N}([g])$, 
we deduce from \eqref{------eq:fourier_pf_b} that
\begin{equation}
\|Z_{(g,1),u,G_N}\|_{\infty}= \sum_{\l\in \widehat{G}_N} d_{\l}^{2-2g}  e^{-\frac u 2 c_\l }=Z_{g,u,G_N}.\label{------eq:Bound Dens gCap}
\end{equation}
Together with the previous lemma, ${\mathscr{R}_{\tilde{M}_U}^{M}}_*(\YM_{M,a,G_N})$  has a density with respect to $\YM_{\tilde M_U,a_U,G_N}$ which is bounded from above by 
\[
\frac{Z_{g,|a|-|a_U|,G_N}}{Z_{g,|a|,G_N}}.
\]
According to Theorem \ref{__THM:PF}, the right-hand side is uniformly bounded in $N$. In particular for any $\varepsilon>0$  and any loop $\ell\in\mathrm{L}(\tilde{M}_U)$,
\[
\PP_{\YM_{M,a}}(|W_\ell- \Phi_{\tilde M_U,a_U}(\ell)|>\varepsilon)\le  \frac{Z_{g,|a|-|a_U|,G_N}}{Z_{g,|a|,G_N}} \PP_{\YM_{\tilde M_U,a_U}}(|W_\ell- \Phi_{\tilde M_U,a_U}(\ell)|>\varepsilon).
\]
Together with Theorem \ref{__THM: Plane} for the plane, we conclude that the right-hand side converges towards zero as $N\to \infty.$

To conclude it remains to prove the uniformity of the convergence in the set $A$ of $a\in \Delta_M(T)$ with $a(f)>0$ for any face $f$ and $|a_U|<T.$   It is enough to show that for any  faces $f_1,f_2$ of $M$ adjacent in $\tilde M_U$, there is a 
constant $K>0$ such that for all  $a,a'\in A,$  with $a(f)=a'(f)$ for all faces distinct from $f_1,f_2$ and $|a-a'|<T,$ 
\begin{equation}
 |\E_{\YM_{M,a}}[ |W_{\ell}-\Phi_{\tilde M_U,a_U}(\ell)|]-\E_{\YM_{M,a'}}[ |W_{\ell}-\Phi_{\tilde M_U,a'_U}(\ell)|]|\le K|a-a'|^{1/2}.\label{eq: UNIF CV thm}
\end{equation}
When $f_1,f_2$ and $a,a'$ are as above with $a(f_1)<a'(f_1)$, while $e$ is an edge of $\tilde M_U$  between  $f_1$ and $f_2$ with $f_2$ on its right, consider an area weighted map $(M^r,a^r)$ finer than $(M,a),$  with an additional edge $ e^r$ with same endpoints as $e$,  an additional face  $\tilde f$ included in $f_2,$ bounded 
by the simple loop $s=e (e^{r})^{-1}$, with  $a^r(\tilde f)=|a-a'|=2 (a'(f_1)-a(f_1)),$ while all other faces can be identified with a face of $M,$ with  $a^r(f_1)=a(f_1), a^r(f_2)=a'(f_2)$ and $a^r(f)=a(f)=a'(f)$ for all $f\not\in\{\tilde f,f_1,f_2\}$.  Denote by $\ell^r$ the loop of $M^r$ obtained from $\ell$ by replacing all occurrences of $e$ in $\ell$ by $e^r$.    Then under $\YM_{M^r,a^r},$   $W_{\ell^r}$ has  same law as $W_\ell$ under 
$\YM_{M,a'}.$    two paths $\a,\b$ of $M^r$ with $\ell=\a\b$,  $W_{\a s\b}$ has  same law as $W_\ell$ under 
$\YM_{M,a'}.$  Using this identity in law, the left-hand-side of \eqref{eq: UNIF CV thm} equals
\begin{align*}
|\E_{\YM_{M^r,a^r}}[ |W_{\ell}-\Phi_{\tilde M_U,a_U}(\ell)|- |W_{\ell^r }-\Phi_{\tilde M^r_U,a^r_U}(\ell^r )|]|&\le  \E_{\YM_{M^r,a^r}} [|W_{\ell}-W_{\ell^r }|]\\ &\hspace{0.3 cm}+ |\Phi_{\tilde M_U,a_U}(\ell)-\Phi_{\tilde M^r_U,a^r_U}(\ell^r )|.
\end{align*}
To conclude it is enough to prove that 
\begin{equation}
\E_{\YM_{M^r,a^r}} [|W_{\ell}-W_{\ell^r }|]\le K |a-a'|^{\frac 1 2},\label{Uniform WL edge deform}
\end{equation}
for some $K>0$ independent of $N.$  Now by induction on the number of occurrences of $s$ in $\ell^r,$ we can assume that $\ell^r=\a s \b$ and $\ell=\a\beta$ for some paths $\a,\b$ of $M^r.$  
%\[ |\E_{\YM_{M^r,a^r}}[ |W_{\a\b}-\Phi_{\tilde M_U,a_U}(\a\b)|- |W_{\a s\b  }-\Phi_{\tilde M^r_U,a^r_U}(\a s\b )|]|\le  \E_{\YM_{M^r,a^r}} [|W_{\a\b}-W_{\a s \b }|].\]
Denoting by $V$ an open disc of $\Sigma$ associated to $\tilde f,$ we can then rewrite and bound  the latter  left-hand side as
\begin{align*}
 \E_{\YM_{M^r,a^r}} [|\tr(H_{\a\b}-H_{\a s \b })|]&= \E_{\YM_{M^r,a^r}} [|\tr(H_{\b\a }-H_{\b\a } H_s)|]\\
 &\le  \E_{\YM_{M^r,a^r}} [\tr[(H_{\b\a} -H_{s} H_{\b\a})(H_{\b\a} -H_{s} H_{\b\a})^*]]^{1/2}\\
 &\le  \sqrt{2} \E_{\YM_{M^r,a^r}}[1- \Re (\tr( H_s)) ]^{1/2}\\
 &\le \frac{\sqrt{2} Z_{g, T-a^r(\tilde f)}}{Z_{g, T}} \E_{\YM_{\tilde M_V^r,a^r_V}}[1- \Re (\tr( H_s)) ]^{1/2}\\
 &\le  \frac{\sqrt{2} Z_{g, T/2}}{Z_{g, T}} \E_{\YM_{\tilde M_V^r,a_V^r}}[1- \Re (\tr( H_s)) ]^{1/2},
\end{align*}
where we used the same argument as in the first part of the proof for the penultimate inequality, and  $2a^r(\tilde f)=|a-a'|<T$.
Now   under $\YM_{\tilde M^r_V,a_V^r}$,  $H_s$ has same law as a Brownian motion at time $|a-a'|.$ It follows from \cite{Lev5} or \cite{Dah} that for all group series considered, 
there is a constant $c>0$ independent of $N$ with 
\[\E_{\YM_{\tilde M_V^r,a_V^r}}[1- \Re (\tr( H_s))]\le 1-e^{-c|a-a' |}\le c|a-a'|,\]
leading to the required bound \eqref{Uniform WL edge deform}.

\end{proof}

\begin{proof}[Proof of Theorem \ref{__THM: Disc YMCont}] We only sketch here the additional arguments needed for the two convergence claims and refer respectively to \cite{Lev2} and \cite{CDG} for more details.  For any multiplicative function $h\in\mathcal{M}(\mathrm{P}(\Sigma), G_N),$ the composition of its restriction to $\mathrm{P}(U)$  with $\Psi$ defines an element $\Psi(h)$ of $\mathcal{M}(\mathrm{P}(D_R),G_N).$ According to the splitting property of Markovian holonomy fields (property (A3) of  \cite[Def. 3.1.2]
{Lev2}) and Lemma 
\ref{_____Lem:Density dYM},  $\Psi_*(\YM_{\Sigma,G_N})$ has density 
\[
\frac{Z_{(g,1),T-\mathrm{vol}(U)}(H_{\pl D_R})}{Z_{g,T,G_N}}
\]
with respect to $\YM_{D_R,G_N}.$ The first claim then follows with the same argument as for the proof of Theorem \ref{__THM:Disc dYM}. To prove the second claim,  for any  $\psi(\ell)$ with $\ell\in\Ld(\overline D_R),$ consider a sequence of loops piecewise 
geodesic loops,  included in an open disc,  converging towards $\psi(\ell)$ for \eqref{Distance Loops}.  The claim then follows adapting the uniform continuity estimate\footnote{The latter estimate is presented in \cite{CDG} for the plane; we claim it applies to 
any compact surface. See \cite[Sect. 5]{DN} where this argument is detailed for the sphere.}  (11) of  \cite[Sect. 4]{CDG}. 
%\mathcal{M}(\mathrm{P}(D_R),G).$  Thanks to Lemma \ref{_____Lem:Density dYM}, the  push-forward  $\tilde\YM_{D_R}=\Psi_*(\YM_\Sigma)$ of the measure $\YM_\Sigma$ then satisfies the following. Whenever $E_M\subset P(D_R)$ 
%is 
%an embedding of an area-weighted map $(M,a)$ of the plane  with infinite face $f_*,$ whose boundary represents $\pl D_R,$
%$${\mathscr{R}_{E_M}}_*(\tilde\YM_{D_R})(dh)=\frac{\psi_{g,T-\|a\|_1}(h_{\delta f_*})}{Z_{g,T}}\YM_{M,a}(dh).$$
%For any $h\in G,$ denote by $[h]$ its conjugacy class.  The last identity implies that  ${\mathscr{R}_{E_M}}_*(\tilde\YM_{D_R})(dh|[h_{\delta f_*}]) = \YM_{M,a}(dh| [h_{\delta f_*}]).$ Choosing $\Sigma=\overline{D_R}$ Since 1. and 2. of Theorem are satisfied for the 
%boundary condition $H_{\delta D_R}= [a]$
\end{proof}

%Under $YM^{\Gbb',\Sigma}$ has density 
%
%with respect to $YM_{\Gbb',\R^2}$
%where 
%\begin{align*}
%\psi_{g,s}(U)&=\int_{G^{2g}}   p_s(U^{-1} [A_1,B_1]\ldots [A_g,B_g]) dA_1dB_1\ldots dA_g dB_g\\
%%&=\int_{G^{2g}}   p_s(U [A_1,B_1]\ldots [A_g,B_g]) dA_1dB_1\ldots dA_g dB_g\\
%&=  \sum_{\lambda\in\widehat{G}} \frac{\chi_\l(U^{-1})}{d_\lambda^{2g-1}}e^{-\frac{s}{2}c_\lambda}
%\end{align*}
%for any $U\in G_r.$ Since the representation is unitary, $|\chi_\l(U)|\le \chi_\l(1)$ and for any $s>0$ the latter sum absolutely converges with 
%\[
%\|\psi_{g,s}\|_{\infty}\le Z_{g,s,X_N}.
%\]
%Since under $YM_{\Gbb',\R^2,N}$,  $\tr(h_\ell )$ converges in probability towards $\Phi(\ell)$, it follows that for any $\varepsilon>0,$ 
%$$\PP_{\Sigma}(|\tr(h_\ell)- \Phi(\ell)|>\varepsilon)\le  \frac{Z_{g,s,X_N}}{Z} \PP_{\R^2}(|\tr(h_\ell)- \Phi(\ell)|>\varepsilon). $$

\subsection{Simple non-contractible loops}

Consider a simple non-contractible loop $\ell$ on a compact surface $\Sigma$ of genus $g\geq 1$.  Two possibilities occur when cutting $\Sigma$ along $\ell.$ If the surface is cut into two surfaces with exactly one connected boundary component, each with genus at least  $1$\footnote{Indeed, if one of the boundary components has genus 0, it implies that $\ell$ is contractible.}, then the loop is \emph{separating}. Otherwise, the new surface has one connected  component and two boundary components, and the loop is called \emph{non-separating}. Both cases are illustrated in Fig. \ref{fig:gluing} below. We refer to \cite{Sti} for details on these loops. It is now important to understand how the Yang--Mills measure interacts with surgery: we therefore begin with a few results that will help us in the next section.

\begin{figure}[!h]
\centering
\includegraphics[scale=0.5]{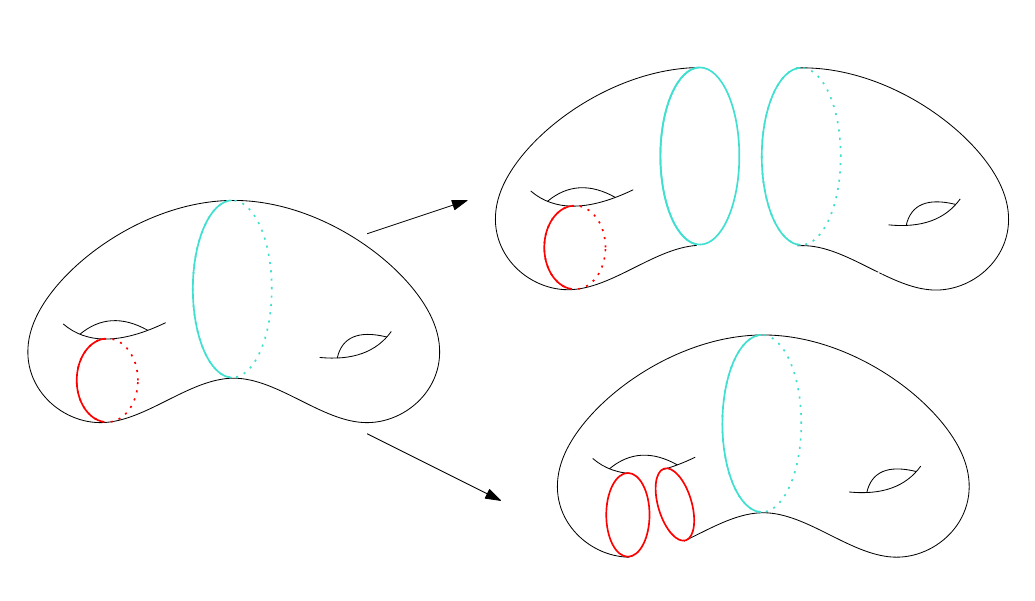}
\caption{\small The surface on the left can be seen as the result of a binary gluing of two surfaces along a separating loop (top right), or of one surface along a nonseparating loop (bottom right).}\label{fig:gluing}
\end{figure}

\subsubsection{Separating loops}

%Consider first two compact connected orientable surfaces $\Sigma_1$ and $\Sigma_2$ such that $\partial\Sigma_1$ and $\partial\Sigma_2$ are connected. Denote by $L_1$ and $L_2$ the corresponding loops, oriented positively. Let $\psi:L_1\to L_2$ be an orientation-reversing diffeomorphism, and $\Sigma$ be the gluing of $\Sigma_1$ and $\Sigma_2$ along $\psi$. Let $L$ be the corresponding loop in the surface $\Sigma$. For $i\in\{1,2\}$, denote by $\tilde{\mathcal{C}}_i$ the sigma-field generated by $(H_{\gamma_1},\ldots,H_{\gamma_n},n\geq 1,\gamma_i\in\mathrm{P}(\Sigma_i))$.  For any $i$, a function $f:\mathcal{M}(\mathrm{P}(\Sigma),G_N)\to\C$ is $\tilde{\mathcal{C}}_i$-measurable if and only if $f\circ\mathscr{R}_{\Sigma_i}:\mathcal{M}(\mathrm{P}(\Sigma_i),G_N)\to\C$ is $\mathcal{C}_i$-measurable. The next theorem is a particular case of \cite[Thm. 5.1.1]{Lev3}.

%
Consider first two compact connected orientable surfaces $\Sigma_1$ and $\Sigma_2$ such that $\partial\Sigma_1$ and $\partial\Sigma_2$ are connected. Denote by $L_1$ and $L_2$ the corresponding loops, oriented positively. Let $\psi:L_1\to L_2$ be an orientation-reversing diffeomorphism, and $\Sigma$ be the gluing of $\Sigma_1$ and $\Sigma_2$ along $\psi$. Let $\ell$ be the corresponding simple loop in the surface $\Sigma$. For $i\in\{1,2\}$, denote by $\tilde{\mathcal{J}}_i$ the sigma-field generated by $(H_{\ell_1},\ldots,H_{\ell_n},n\geq 1,\ell_i\in\mathrm{L}(\Sigma_i))$.  For any $i$, a function $f:\mathcal{M}(\mathrm{P}(\Sigma),G_N)\to\C$ is $\tilde{\mathcal{J}}_i$-measurable if and only if $f\circ\mathscr{R}_{\Sigma_i}:\mathcal{M}(\mathrm{P}(\Sigma_i),G_N)\to\C$ is $\mathcal{J}_i$-measurable. The next theorem is a particular case of \cite[Thm. 5.1.1]{Lev3}; a different version\footnote{applied to non-trivial bundles} is given in  \cite[Thm.  1]{Sen6} (and informally in (1)).

\begin{thm}\label{thm:split_separ}
The sigma-fields $\tilde{\mathcal{J}}_1$ and $\tilde{\mathcal{J}}_2$ are independent on $\mathcal{M}(\mathrm{P}(\Sigma),G_N)$ under $\YM_{\Sigma,G_N}$ conditionally on the random variable $[H_\ell]$. Moreover, for any $\tilde{\mathcal{J}}_i$-measurable function $f_i:\mathcal{M}(\mathrm{P}(\Sigma),G_N)\to\C$, for $i\in\{1,2\}$, the product $f_1f_2$ is measurable with respect to $\mathcal{J}$ and  the following equality holds true for any $t\in G_N/\mathrm{Ad}$:
\begin{align}
\int f_1(h)f_2(h)&\YM_{\Sigma,C_{\ell\mapsto t},G_N}(dh)\nonumber\\ &=\int f_1\circ\mathscr{R}_{\Sigma_1}(h)\YM_{\Sigma_1,t,G_N}(dh)\int f_2\circ\mathscr{R}_{\Sigma_2}(h)\YM_{\Sigma_2,t^{-1},G_N}(dh).
\end{align}
\end{thm}

This theorem provides a sort of spacial Markov property, which is a consequence of the semigroup property of the heat kernel. It has the following application.

\begin{coro}
Let $\ell$ be a separating loop in a closed, area-weighted map $(M,a)$ of genus $g\geq 2$, splitting $M$ into two respective maps $M_1$ and $M_2$ of genus $g_1,g_2$ and total area $T_1$ and $T_2$. Let $f:\mathcal{M}(\mathrm{P}(M),G_N)\to\C$ be a bounded $\tilde{\mathcal{J}}_1$-measurable function. We have
\begin{equation}
\int_{\mathcal{M}(\mathrm{P}(M),G_N)}f(h)\YM_{M,a,G_N}(dh)=\int_{\mathcal{M}(\mathrm{P}(M_1),G_N)} f\circ\mathscr{R}_{M_1}^M(h)I(h_\ell)\YM_{M_1,a,G_N}(dh),\label{-----eq:Disint one side}
\end{equation}
where for any $x\in G_N$,
{
\[
I(x)= \frac{ Z_{(g_2,1),T_2,G_N}(x^{-1})}{Z_{g,T,G_N}}.
\]}
\end{coro}

\begin{proof}
To lighten the notation, we will drop in this proof the subscripts $G_N$, as the structure group remains fixed. First of all, using \eqref{eq:decondition}, the left-hand-side of \eqref{-----eq:Disint one side} equals 
 \begin{align*}
\frac{1}{Z_{g,T}}\int_{G/\mathrm{Ad}}\int_{\mathcal{M}(\mathrm{P}(M),G_N)}f(h) Z_{g,T}(\ell;t)\YM_{M,C_{\ell\mapsto t},a}(dh)dt.
\end{align*}
%\int_{\mathcal{M}(\mathrm{P}(M),G_N)}f(h)\YM_{M,a}(dh)=
%Here $ Z_{g,T}(t)$ is  not given by Prop. \ref{______Prop:fourier_pf}, because here the condition is not on a boundary but on an inner loop. 

Let us  compute $Z_{g,T}(\ell;t)$. Using the invariance property of Yang--Mills measure by subdivision, one can assume without loss of generality that that for $i\in\{1,2\},$ $M_i$ has 1 vertex $v^{(i)}$, $2g_i+1$ edges $a_1^{(i)},b_1^{(i)},\ldots,a_g^{(i)},b_g^{(i)},e^{(i)}$ all with $v^{(i)}$ as source and target, and 1 face whose boundary is given by ${e^{(i)}}^{\varepsilon_i}[a_1^{(i)},b_1^{(i)}]\cdots[a_{g_i}^{(i)},b_{g_i}^{(i)}]$ where $\varepsilon_1=1$ and $\varepsilon_2=-1$. The map $M$ can then be defined as the quotient of $M_1\cup M_2$ by the relations $v^{(1)}\sim v^{(2)}$ and $(e^{(1)})^{-1}\sim e^{(2)}$. We obtain
\begin{align*}
Z_{g,T}(\ell;t)= & \int_{G_N^{2g+1}}p_{T_1}(x[y_1,z_1]\cdots[y_{g_1},z_{g_1}])\\
&\times p_{T_2}(x^{-1}[y'_1,z'_1]\cdots[y'_{g_2},z'_{g_2}])dx\prod_{\substack{1\leq i\leq g_1\\ 1\leq j\leq g_2}} dy_idz_idy'_jdz'_j\\
= & Z_{(g_1,1),T_1}(t)Z_{(g_2,1),T_2}(t^{-1}).
\end{align*}
%where the  Prop. \ref{______Prop:fourier_pf}. 
From this result and from Theorem \ref{thm:split_separ} we deduce
\begin{align*}
& \int_{\mathcal{M}(\mathrm{P}(M),G_N)}f(h)\YM_{M,a}(dh)\\
& = \frac{1}{Z_{g,T}}\int_{G/\mathrm{Ad}}\Bigg[\int_{\mathcal{M}(\mathrm{P}(M_1),G_N)}f\circ\mathscr{R}_{M_1}^M(h)Z_{(g_1,1),T_1}(t)\YM_{M_1,t,a}(dh)\\
& \times \int_{\mathcal{M}(\mathrm{P}(M_2),G_N)}Z_{(g_2,1),T_2}(t^{-1})\YM_{M_2,t^{-1},a}(dh)\Bigg]dt.
\end{align*}
The integral over $\mathcal{M}(\mathrm{P}(M_2),G_N)$ is simply equal to $Z_{(g_2,1),T_2}(t^{-1})$; furthermore, under $\YM_{M_1,t,a}$, $[H_\ell]=t,$ so that we have
\begin{align*}
&Z_{g,T}  \int_{\mathcal{M}(\mathrm{P}(M),G_N)}f(h)\YM_{M,a}(dh)\\
& = \int_{G/\mathrm{Ad}}\int_{\mathcal{M}(\mathrm{P}(M_1),G_N)}f\circ\mathscr{R}_{M_1}^M(h) Z_{(g_1,1),T_1}(h_\ell)Z_{(g_2,1),T_2}(h_\ell^{-1})\YM_{M_1,t,a}(dh)dt.
\end{align*}
Finally, using the disintegration formula \eqref{-----eq: Disint YM BD} leads to the expected equality.
% \eqref{eq:decondition} gives the expected equality.
\end{proof}

To prove Theorem \ref{__THM:simple non-contrac}, we shall need the following  similar but simpler lemma.

\begin{lem}\label{cor:Density_sep}
Let $\ell$ be a separating loop in a closed, area-weighted map $(M,a)$ of genus $g\geq 2$, splitting $M$ into two respective maps $M_1$ and $M_2$ of genus $g_1,g_2$ and total area $T_1$ and $T_2$. Under $\YM_{M,a,G_N}$, the random variable $H_\ell$ has density
\begin{equation*}
\frac{Z_{(g_1,1),T_1,G_N} (h)Z_{(g_2,1),T_2,G_N} (h^{-1})}{Z_{g,T,G_N}}
\end{equation*}
with respect to the Haar measure on $G_N.$
%, where  for all $g\ge 1, s>0$,  $\psi_{g,s}$ is defined by $\psi_{g,s,G_N}:h\mapsto Z_{g,s,G_N}([h])$.
\end{lem}

%\begin{proof}
%1. Let $f:G\to\C$ be a bounded central measurable function. It corresponds to a measurable function $\tilde{f}:G/\mathrm{Ad}\to\C$. We have
%\[
%\E[f(H_\ell)]=\int_{\mathcal{M}(P(M),G} f(h_\ell)Z_{g_1,\|a\|}(t)\YM_{M,a,G}(dh),
%\]
%which yields
%\begin{align*}
%\E[f(H_\ell)]=&\frac{1}{Z_{g,\|a\|}}\int_{G/\mathrm{Ad}}\int_{\mathcal{M}(P(M),G} f(h_\ell)Z_{g_1,\|a\|}(t)\YM_{M,a,G}(t)(dh)dt\\
%= &\frac{1}{Z_{g,\|a\|}}\int_{G/\mathrm{Ad}} \tilde{f}(t)Z_{g_1,\|a\|}(t)\int_{\mathcal{M}(P(M),G} \YM_{M,a,G}(t)(dh)dt.
%\end{align*}
%The intgral over $\mathcal{M}(\mathrm{P}(M),G)$ is equal to 1 and we are left with
%\[
%\frac{Z_{g,\|a\|}(t)}{Z_{g,\|a\|}}.
%\]
%Note that the partition function has not the same meaning here than in Prop. \ref{______Prop:fourier_pf} because the condition is not on the boundary. However we can compute it as follows:
%
%2. 
%\end{proof}

\begin{proof}
Let us use the same $M,M_1$ and $M_2$ as in the previous proof. For any central function $f:G_N\to\C$ measurable and bounded, we have
\begin{align*}
\E_{\YM_{M,a,G_N}}[f(H_\ell)]= &\frac{1}{Z_{g,T,G_N}}\int_{G_N^{2g+1}} f(x)p_{T_1}(x[y_1,z_1]\cdots[y_{g_1},z_{g_1}])\\
& \times p_{T_2}(x^{-1}[y'_1,z'_1]\cdots[y'_{g_2},z'_{g_2}])dx\prod_{\substack{1\leq i\leq g_1\\ 1\leq j\leq g_2}} dy_idz_idy'_jdz'_j,
\end{align*}
thus by integrating over all variables but $x$ we obtain the expected result.
\end{proof}

\begin{proof}[Proof of Theorem \ref{__THM:simple non-contrac} for separating, non-contractible, simple loops]  Using Lemma \ref{cor:Density_sep} and the bound \eqref{------eq:Bound Dens gCap}, we find that under $\YM_{M,a}$, the density  of  $h_\ell$ with respect to the Haar measure is uniformly bounded by
$$\frac{Z_{g_1,T_1,G_N}Z_{g_2,T_2,G_N}}{Z_{g,T,G_N}}.$$
Since $\ell$ is not contractible, $g_1,g_2\ge 1,$ and according to Theorem \ref{__THM:PF}, the above density is uniformly bounded in $N.$ Now for a Haar distributed random variable $H$ on $G_N$ belonging to $X_r,$ according to \cite{DiaconisEvans}, 
whenever $n\not=0,$ $\tr(H^n)$ converges in probability towards $0$, as $N\to \infty.$ The claim now 
follows by absolute continuity similarly to 
the proof of  Theorem \ref{__THM:Disc dYM}.
\end{proof}
%
%Using again that $ \frac{|\chi_\l|}{\chi_\l(1)}$ is uniformly bounded by $1$, we get that 
%$$ \Psi_{t,T,N}(u)\le  \frac{Z_{t,g_1} Z_{T-t,g_2}}{Z_T}.$$
%Thanks to Theorem **** for the partition functions, it follows that
%$$\limsup_N\|\Psi_{t,T,N}\|_{\infty}<+\infty.$$
%Now, since  the claim holds true   for Haar measure on the considered groups  REF (see e.g. Collins $\&$ Sniady/Diaconis), the result follows. 

\subsubsection{Non-separating  loops}
\label{sec:Non-Sep Loop}

Let $\Sigma'$ be a compact connected orientable surface such that $\partial\Sigma'$ has two connected components; denote the corresponding loops, oriented positively, by $L_1$ and $L_2$. Let $\psi:L_1\to L_2$ be an 
orientation-reversing diffeomorphism, and $\Sigma$ be the gluing of $\Sigma'$ along $\psi$. Let $\ell$ be the corresponding loop in $\Sigma$. We have the following, which is a particular case of \cite[Thm. 5.4.1]{Lev3}. We use the same notation $\tilde{\mathcal{J}}'$ for the sigma-field on $\mathcal{M}(\mathrm{P}(\Sigma),G)$ generated by the loops in $\Sigma'$.

\begin{thm}
Let $f:\mathcal{M}(\mathrm{P}(\Sigma),G)\to\C$ be a $\tilde{\mathcal{J}}'$-measurable function. For any $t\in G/\mathrm{Ad}$, the following equality holds true:
\begin{equation}
\int f(h)\YM_{\Sigma,C_{\ell\mapsto t},G}(dh) = \int f\circ\mathscr{R}_{\Sigma'}(h)\YM_{\Sigma',(t,t^{-1}),G}(dh).
\end{equation}
\end{thm}

The consequence in terms of absolute continuity is the following.

\begin{coro}
Let $\ell$ be a non-separating simple loop in a closed, area-weighted map $(M,a)$ of genus $g\geq 1$. Under $\YM_{M,a,G}$, the random variable $H_\ell$ has density given by
\begin{equation*}
\frac{\varphi_{g,T}(h)}{Z_{g,|a|,G}},\forall h\in G
\end{equation*}
 with respect to the Haar measure, and  for all $g\ge 1, T>0$,  
\begin{equation}
\varphi_{g,T}(h)=\sum_{\l\in \widehat{G}}d_\l^{2-2g} |\chi_\l(h)|^2e^{-\frac{T}{2}c_\l},\forall h\in G. \label{------eq:DensYM NonS Loop}
\end{equation}
\end{coro}
\begin{rmk} The above density can also be written as $Z_{(g,2),|a|}(h,h^{-1}).$
\end{rmk}
%Assume that $\ell$ is a non-separating, simple, non-contractible loop within an area-weighted map $(M,a)$ of genus $g\ge 1$.
%  
%\begin{lem} \label{_____Lem:Density Sep dYM} Under $\YM_{M,a}$, the random variable  $h_\ell$ has density given by
%\begin{equation*}
%\frac{\varphi_{\|a\|_1,g}(h)}{Z_{g,\|a\|_1,N}},\forall h\in G
%\end{equation*}
% with respect to the Haar measure, where  for all $g\ge 1, T>0$,  
%\begin{equation}
%\varphi_{g,T}(h)=\sum_{\l\in \widehat{G}}d_\l^{2-2g} |\chi_\l(h)|^2e^{-\frac{T}{2}c_\l},\forall h\in G. \label{------eq:DensYM NonS Loop}
%\end{equation}
%  \end{lem}
\begin{proof}  According to \cite[Sect. 6.3.1]{Sti}, or \cite[Sect. 1.3.1]{FM}, since $\ell$ is not separating, we can choose $M$ with $1$ vertex, $1$ face and $2g$ edges, one of which is $\ell$, and  writing $u=h_\ell,$ 
\begin{equation}
\YM_{M,a}(dh)=p_{T}([u,v][x_2,y_2]\cdots [x_g,y_g]) dudv \prod_{i=2}^{g}dx_idy_i.\label{------eq:DensYM NonS Loop}
\end{equation}
It follows that the law of $h_\ell$ under $\YM_{M,a}$ has density 
$$\varphi_{g,T}(u)=Z_{g,T,N}^{-1}\int_{G^{2g-1}} p_{T}([u,v][x_2,y_2]\cdots [x_g,y_g]) dv \prod_{i=2}^{g}dx_idy_i,\forall u\in G, $$
with respect to the Haar measure. Expanding in characters and using \eqref{------eq:int_commu} yields
$$\varphi_{g,T}(h)=\sum_{\l\in \hat G}e^{-\frac{T}{2}c_\l} I_\l(u) $$
with $$I_\l(u)= d_\l^{1-2(g-1)} \int_{G} \chi_\l([u,v])dv.  $$
Now \eqref{------eq:int_conj} yields $I_\l(u)= d_\l^{2-2g}\chi_\l(u)\chi_\l(u^{-1}), \forall u\in G$ and the claim.
\end{proof}

\begin{proof}[Proof of Theorem \ref{__THM:simple non-contrac} for $g\ge 2$ and non-separating, simple loops] Then, using that $\|\chi_\l \|_\infty=\chi_\l(1)$ for all $\l\in\widehat{G},$ 
$$\|\varphi_{g,T}\|_\infty= \varphi_{g,T}(1)=Z_{g-1,T,G_N},\forall g\ge 1,T>0.$$
It follows that under $\YM_{M,a}$, the density of $h_\ell$ with respect to the Haar measure is bounded by 
$$\frac{Z_{g-1,|a|,G_N}}{Z_{g,|a|,G_N}}.$$
When $g\ge 2,$ it follows from \ref{__THM:PF} that this sequence is uniformly bounded when $N\to \infty$. As  in the argument of the former section for the separating case, the claim follows from \cite{DiaconisEvans} and absolute continuity. 
\end{proof}  

When $g=1,$ the above argument of absolute continuity argument fails. Indeed for a total area $T,$ the maximum of the density is then given by $\frac{Z_{0,T,N}}{Z_{1,T,N}}.$ For $\tilde A_N,$ thanks to Theorem \ref{__THM: DK 
PT}, 
and more precisely formula (28) of \cite{LM},
$$\lim_{N\to\infty}\frac{1}{N^2}\log (Z_{0,T,N})= F(T)=\frac{T}{24}+\frac 3 4-\frac 1 2 \log(T),\forall 0<T\le \pi^2.$$
Since  $F(T)\ge F(\pi^2)>0$ for all $T\in(0,\pi^2],$  $\lim_{N\to \infty} Z_{0,T,N}= +\infty.$ 
Similarly to the strategy of \cite{Lev,DN,Hal2},  we shall consider the expectation and the variance of  the random variable $W_\ell.$ 
\begin{rmk} For any $z$ in the center of $G, $ under $\YM_{M,a},$ $z h_\ell$ has same law as $h_\ell.$  This can be checked changing variable in formula \eqref{------eq:DensYM NonS Loop}. Alternatively, Schur's lemma implies that 
$z$ 
acts  by 
a unitary scalar in any irreducible representation and \eqref{------eq:DensYM NonS Loop} implies that for any $\varphi_{g,T}(zh)=\varphi_{g,T}(h)$ for any $h\in G.$ Consequently whenever $D_z\in G$ is a diagonal matrix of multiplication 
by a 
scalar  $z\in\C,$ under $\YM_{M,a},$ $\tr(h_{\ell^n})=\tr(h_\ell^n)$  has same law as $z^n \tr(h_{\ell^n}).$ For $X_N=\{A_N,\tilde A_N\},$ the center is given by $\{\zeta \mathrm{Id}_{N+1} : \zeta^{N+1}=1\}$ and $\{z \mathrm{Id}_N: 
z\in\C, |z|=1\}.$  It follows that for $X=\tilde A,$ $\E[W_{\ell^n}]=0$ for any $n\not=0$, whereas for $X=A,$ $\E[W_{\ell^n} ]= 0$  for any $n\not=0$ modulo $N+1$. In particular, when $X\in\{A,\tilde A\},$ $\lim_{N\to\infty}
\E[W_{\ell^n}]=0.$
\end{rmk}

Unfortunately, the argument given in the above remark does not apply when $X=\{B,C,D\}$. Also, it does not give any information about $\mathrm{Var}(W_\ell)=\E[|W_\ell|^2]-|\E[W_\ell]|^2.$ Instead, we shall use the expansion in 
characters of the heat kernel 
and the following lemma, which is a consequence of the expression of characters as a ratio of alternated functions.
\begin{lem}\label{_____Lem:Pieri Gen} Let $G_N\subset\GL_n(\C)$ be a classical group of size $N$ and type $X_r$, for $X\in\{B,C,D\}$. For any $k\neq0$ and $\l \in \widehat{G}_N,$
\begin{equation}\label{eq:Pieri_BCD}
\Tr(g^k)\chi_\l(g)=\sum_{\mu\in\widehat{G}_N} c_{\l,k}^\mu \chi_{\mu}(g), 
\end{equation}
with
\begin{equation}\label{------eq:Loop  jumps}
c_{\l,k}^\mu \in \{-1,0,1\},\forall \mu\in\widehat{G}  _N
\end{equation}
and
\begin{equation}
\sum_{\mu\in\widehat{G}_N} |c_{\l,k}^\mu|\le n\label{------eq:Bound number  possible jumps}
\end{equation}
\end{lem}

\begin{proof}
Let us introduce a few common notations for the different cases. Set $\Z_{\sym}=\Z$ if $r$ is odd or $\Z+\tfrac12$ if $r$ is even. The mapping
\[
\lambda\in\widehat{G}_N\mapsto\lambda+\rho
\]
establishes a bijection between the set of highest weights and $\{\mu\in\Z_\sym^r:\mu_1>\cdots>\mu_r\}$. The symmetric group $\mathfrak{S}_r$ acts in an obvious way on
\[
\Delta_r=\{\mu\in\Z_\sym^r:\mu_1>\cdots>\mu_r\}\cong\{\mu\in\Z_\sym^r:\mu_i\neq\mu_j,\ \forall i\neq j\}/\mathfrak{S}_r,
\]
where the symmetric group $\mathfrak{S}_r$ acts on $\mu$ by permutation of its components. For  $\mu$ with $\mu_i\not=\mu_j,\forall i\not=j,$  we denote by $[\mu]$ 
its decreasing rearrangement and $\sigma_\mu\in \mathfrak{S}_r$ the unique permutation such that $[\mu]_i=\mu_{\sigma(i)}$ for all $i$. We will prove each case separately using the previous notations. 
The explicit formulae for the characters $\chi_\lambda$ 
that we will use can be found in \cite{BtD} or in \cite[Sect. 2.3]{Mel}. When $\mu\in \Z_{\sym}^r$ with $\mu_i=\mu_j$ for some $i\not=j$, by convention $\chi_{[\mu]-\rho}=0$ and $\varepsilon(\sigma_{\mu})=0.$ The proof for all cases will rely on the following 
computation: for any $(z_1,\ldots,z_r)\in(\C^*)^r,$ any $(m_1,\ldots,m_r)\in\Z^r$ and any $k\in\N^*$, we have
\begin{equation}\label{eq:det}
\sum_{\ell=1}^r(z_\ell^k+z_\ell^{-k})\det(z_i^{m_j}-z_i^{-m_j})=\sum_{\ell=1}^r(\det(z_i^{m_j+k\delta_{j\ell}})+\det(z_i^{-m_j-k\delta_{j\ell}})).
\end{equation}

Let us prove \eqref{eq:det}. Using the Leibniz formula for determinant and the invariance of $\sum_\ell(z_\ell^k+z_\ell^{-k})$ by the action $\ell\mapsto \sigma(\ell)$ of the symmetric group, we have
\begin{align*}
\sum_{\ell=1}^r(z_\ell^k+z_\ell^{-k})\det(z_i^{m_j}-z_i^{-m_j}) = & \sum_{\sigma\in\mathfrak{S}_r}\ve(\sigma)\sum_{\ell=1}^r(z_\ell^k+z_\ell^{-k})\prod_{i=1}^r (z_{\sigma(i)}^{m_i}-z_{\sigma(i)}^{-m_i})\\
= & \sum_{\sigma\in\mathfrak{S}_r}\ve(\sigma)\sum_{\ell=1}^r(z_{\sigma(\ell)}^k+z_{\sigma(\ell)}^{-k})\prod_{i=1}^r (z_{\sigma(i)}^{m_i}-z_{\sigma(i)}^{-m_i}).
\end{align*}
We can then put the sum over $\mathfrak{S}_r$ back in by linearity, and for any $\sigma$ and any $\ell$ we have
\[
(z_{\sigma(\ell)}^k+z_{\sigma(\ell)}^{-k})\prod_{i=1}^r(z_{\sigma(i)}^{m_i}-z_{\sigma(i)}^{-m_i}) = (z_{\sigma(\ell)}^{m_\ell+k}-z_{\sigma(\ell)}^{-m_\ell-k}+z_{\sigma(\ell)}^{m_\ell-k}-z_{\sigma(\ell)}^{-m_\ell+k})\prod_{i\neq \ell}(z_{\sigma(i)}^{m_i}-z_{\sigma(i)}^{-m_i}).
\]
Summing over $\ell$ leads to \eqref{eq:det}. We can now prove \eqref{eq:Pieri_BCD} for each type of group.

Case $B_r$: let $g\in\SO(2r+1)$ be an element with eigenvalues $(z_1^{\pm 1},\ldots,z_r^{\pm 1},1)$. We have
\[
\chi_\lambda(g)=\frac{\det(z_i^{\lambda_j+\rho_j}-z_i^{-\lambda_j-\rho_j})}{\det(z_i^{\rho_j}-z_i^{-\rho_j})}.
\]
If we substitute $\mu=\lambda+\rho\in\Delta_r,$ we can write
\[
\chi_\lambda(g)=\chi_{\mu-\rho}(g)=\frac{\det(z_i^{\mu_j}-z_i^{-\mu_j})}{\det(z_i^{\rho_j}-z_i^{-\rho_j})},
\]
so that
\[
\Tr(g^k)\chi_\lambda(g)= \left(1+\sum_{\ell=1}^r (z_\ell^k+z_\ell^{-k})\right)\frac{\det(z_i^{\mu_j}-z_i^{-\mu_j})}{\det(z_i^{\rho_j}-z_i^{-\rho_j})}
\]
and by \eqref{eq:det} we obtain that
\begin{equation}
\Tr(g^k)\chi_\lambda(g)= \chi_\lambda(g) + \sum_{\ell=1}^r \frac{\det(z_i^{\mu_j+k\delta_{j\ell}}-z_i^{-(\mu_j+k\delta_{j\ell})})+\det(z_i^{\mu_j-k\delta_{j\ell}}-z_i^{-(\mu_j-k\delta_{j\ell})})}{\det(z_i^{\rho_j}-z_i^{-\rho_j})}\label{-----eq: PieriB}
\end{equation}
Set $\mu_{k,\ell}^+=(\mu_1,\ldots,\mu_\ell+k,\ldots,\mu_r)$ and $\mu_{k,\ell}^-=(\mu_1,\ldots,\mu_\ell-k,\ldots,\mu_r)$. We have
\begin{align*}
\Tr(g^k)\chi_\lambda(g) = & \chi_\lambda(g) + \sum_{\ell=1}^r \left(\ve(\sigma_{\mu_{k,\ell}^+})\chi_{[\mu_{k,\ell}^+]-\rho}(g)+\ve(\sigma_{\mu_{k,\ell}^-})\chi_{[\mu_{k,\ell}^-]-\rho}(g)\right),
\end{align*}
which yields \eqref{eq:Pieri_BCD} with $c_{\lambda,k}^\lambda=1$, $\ c_{\lambda,k}^{[\mu_{k,\ell}^+]-\rho}=\ve(\sigma_{[\mu_{k,\ell}^+]})$ for all $1\leq \ell\leq r$ and $c_{\lambda,k}^\mu=0$ for any other $\mu\in\widehat{G}_N$, which also proves \eqref{------eq:Loop  jumps}. Indeed, each $[\mu_{k,\ell}^\pm]-\rho$ corresponds to an element of $\widehat{G}_N$ if and only if $[\mu_{k,\ell}^\pm]\in\Delta_r$, and if it is not the case, then $\chi_{[\mu_{k,\ell}^\pm]}$ and summand in the right-hand-side of \eqref{-----eq: PieriB} vanish. Finally, there are $2r+1$ nonzero coefficients, whose absolute value is always equal to 1, hence \eqref{------eq:Bound number  possible jumps} is also satisfied.

Case $C_r$: let $g\in\Sp(r)$ be an element with eigenvalues $(z_1^{\pm 1},\ldots,z_r^{\pm 1})$. We have again
\[
\chi_\lambda(g)=\chi_{\mu-\rho}(g)=\frac{\det(z_i^{\mu_j}-z_i^{-\mu_j})}{\det(z_i^{\rho_j}-z_i^{-\rho_j})}.
\]
It follows that
\[
\Tr(g^k)\chi_\lambda(g) = \sum_{\ell=1}^r \left(\ve(\sigma_{\mu_{k,\ell}^+})\chi_{[\mu_{k,\ell}^+]-\rho}(g)+\ve(\sigma_{\mu_{k,\ell}^-})\chi_{[\mu_{k,\ell}^-]-\rho}(g)\right),
\]
as expected.

Case $D_r$: let $g\in\SO(2r)$ be an element with eigenvalues $(z_1^{\pm 1},\ldots,z_r^{\pm 1})$. We have
\[
\chi_\lambda(g)=\chi_{\mu-\rho}(g)=\frac{\det(z_i^{\mu_j}-z_i^{-\mu_j})+\det(z_i^{\mu_j}+z_i^{-\mu_j})}{\det(z_i^{\rho_j}+z_i^{-\rho_j})},
\]
therefore
\begin{align*}
\Tr(g^k)\chi_\lambda(g)= & \left(\sum_{\ell=1}^r (z_\ell^k+z_\ell^{-k})\right)\frac{\det(z_i^{\mu_j}-z_i^{-\mu_j})+\det(z_i^{\mu_j}+z_i^{-\mu_j})}{\det(z_i^{\rho_j}+z_i^{-\rho_j})}\\
= & \sum_{\ell=1}^r \left(\ve(\sigma_{\mu_{k,\ell}^+})\chi_{[\mu_{k,\ell}^+]-\rho}(g)+\ve(\sigma_{\mu_{k,\ell}^-})\chi_{[\mu_{k,\ell}^-]-\rho}(g)\right).
\end{align*}
%
%$======================$
%
%Let $g\in\SO(2N+1)$ be an element with eigenvalues $(z_1^{\pm 1},\ldots,z_N^{\pm 1},1)$. We have $W=(\Z/2\Z)^{2N+1}\ltimes S_{2N+1}$ and therefore
%
%\begin{align*}
%\Tr(g^k)\chi_\lambda(g)= & \left(1+\sum_{i=1}^N (z_i+z_i^{-1})\right)\frac{\sum_{\sigma\in S_{2N+1}}\sum_{\eta\in(\Z/2\Z)^{2N+1}}\varepsilon(\sigma\cdot\eta)z_i^{(\lambda+\rho)_{\sigma\cdot\eta(i)}} }{\sum_{\sigma\in S_{2N+1}}\sum_{\eta\in(\Z/2\Z)^{2N+1}}\varepsilon(\sigma\cdot\eta)z_i^{\rho_{\sigma\cdot\eta(i)}}}\\
%= & \left(1+\sum_{i=1}^N (z_i+z_i^{-1})\right)
%\end{align*}

\end{proof}

\begin{proof}[Proof of Theorem \ref{__THM:simple non-contrac} for $g=1$ and non-separating, simple loops] Let us set $T=|a|.$ We shall compute the expectation and the variance of $W_\ell$  under $\YM_{M,a,G_N},$ and show that both converge to zero as $N\to\infty.$ Assume that $G_N$ is  of type $X_r$ and a subgroup of $\mathrm{GL}_n(\C)$ as\footnote{so that $n$ is respectively  $r,r+1,2r+1$ for types $\tilde A_r,A_r,B_r$ and $2r$ for types $\{C_r,D_r\}.$ } in section \ref{sec:Intro Class Gp}.   
Using Corollary \ref{cor:Density_sep} and expanding in characters yield
\begin{align*}
n\E[W_{\ell^k}]&= Z_{1,T,X_r}^{-1}\int_{G_N} \varphi_{T,1}(h) \Tr(h^k)dh\\
&= Z_{1,T,X_r}^{-1}\sum_{\l \in\widehat{G}_N} e^{-\frac{T}{2}c_\l}  I_\l,
\end{align*}
where  for all $\l\in \widehat{G}_N,$ using the orthogonality of characters,
$$I_\l=\int_{G_N} \Tr(h^k) \chi_\l(h)\chi_{\l}(h^{-1})dh=\sum_{\mu\in \widehat{G}_N} c_{\l,k}^\mu \int_{G_N} \chi_\mu(h)\chi_\l(h^{-1})dh=c_{\l,k}^\l.  $$
Thanks to \eqref{------eq:Loop  jumps}, it follows that 
\begin{equation}
|\E[W_{\ell^k}]|\le (n Z_{1,T,X_r})^{-1} \sum_{\l\in\widehat{G}_N} e^{-\frac T 2 c_\l}=\frac{1}{n}. \label{------eq:Bound E}
\end{equation}
Similarly, 
\begin{align*}
n^2\E(|W_{\ell^k}|^2)&=Z_{1,T,X_r}^{-1}\int_{G_N} |\Tr(h^k)|^2  \varphi_{T,1}(h)dh= Z_{1,T,X_r}^{-1}\sum_{\l\in\widehat{G}_N} e^{-\frac T 2 c_\l} J_\l, 
\end{align*}
where 
\begin{align*}
J_\l&=\int_{G_N} \Tr(h^k)\chi_\l(h)\Tr(h^{-k})\chi_\l(h^{-1})dh\\
&=\sum_{\mu,\nu} c_{\l,k}^\mu c_{\l,k}^\nu \int_{G} \chi_\mu(h)\chi_\nu(h^{-1})dh= \sum_{\mu\in\widehat{G}_N} (c_{\l,k}^\mu)^2. 
\end{align*}
Using Lemma \ref{_____Lem:Pieri Gen}, we conclude that 
$$0\le J_\l=\sum_{\mu\in\widehat{G}_N} |c_{\l,k}^\mu|\le n$$
and
\begin{equation}
\E[|W_{\ell^k}|^2]\le n^{-2}Z_{1,T,X_r}^{-1}\sum_{\l\in\widehat{G}_N} e^{-\frac T 2 c_\l} J_\l= \frac{1}{n}.\label{------eq:Bound V}
\end{equation}
From \eqref{------eq:Bound E} and \eqref{------eq:Bound V}, we conclude that the expectation and variance of $W_{\ell^k}$ under $\YM_{M,a,G_N}$ both vanish as $N\to\infty.$  Therefore, $W_{\ell^k}$ converges to $0$ in probability.
\end{proof}

\begin{rmk} For the $A_r$ type, for all $k\neq 0$ and $\l\in \L_r,$ $c_{\l,k}^\l=0,$ and the first part of the above proof yields another argument for $\E[W_{\ell^k}]=0$ for all $r\ge 1.$ 
\end{rmk}

\section*{Acknowledgments} Both authors are grateful to Thierry L\'evy for several discussions which inspired the main content of this article and to B. Hall for his feedback on a first draft of this work.  Many  thanks are also due to Neil O'Connell, as this joint work  started during a visit of 
the 
second author to the first one at  UCD and was founded by the  ERC Advanced Grant 669306.    A.D. acknowledges partial support from  "The Dr Perry James (Jim) Browne Research Centre on Mathematics and its Applications" 
individual grant. T.L. acknowledges partial support from ANR AI chair BACCARAT (ANR-20-CHIA-0002).  Many thanks are due to two anonymous referees whose careful reading and comments greatly improved the first version of this text.

\bibliography{Biblio}

\begin{thebibliography}{10}

\bibitem{AGZ}
Greg~W. Anderson, Alice Guionnet, and Ofer Zeitouni.
\newblock {\em An introduction to random matrices}, volume 118 of {\em
  Cambridge Studies in Advanced Mathematics}.
\newblock Cambridge University Press, Cambridge, 2010.

\bibitem{AS}
Michael Anshelevich and Ambar~N. Sengupta.
\newblock Quantum free {Y}ang--{M}ills on the plane.
\newblock {\em J. Geom. Phys.}, 62(2):330--343, 2012.

\bibitem{Bia}
Philippe Biane.
\newblock Free {B}rownian motion, free stochastic calculus and random matrices.
\newblock In {\em Free probability theory ({W}aterloo, {ON}, 1995)}, volume~12
  of {\em Fields Inst. Commun.}, pages 1--19. Amer. Math. Soc., Providence, RI,
  1997.

\bibitem{BMS}
A.~Boutet~de Monvel and M.~V. Shcherbina.
\newblock On free energy in two-dimensional {${\rm U}(n)$}-gauge field theory
  on the sphere.
\newblock {\em Teoret. Mat. Fiz.}, 115(3):389--401, 1998.

\bibitem{BtD}
Theodor Br\"{o}cker and Tammo tom Dieck.
\newblock {\em Representations of compact {L}ie groups}, volume~98 of {\em
  Graduate Texts in Mathematics}.
\newblock Springer-Verlag, New York, 1995.

\bibitem{CDG}
Guillaume C\'{e}bron, Antoine Dahlqvist, and Franck Gabriel.
\newblock The generalized master fields.
\newblock {\em J. Geom. Phys.}, 119:34--53, 2017.

\bibitem{Chatt}
Sourav Chatterjee.
\newblock Rigorous solution of strongly coupled so(n) lattice gauge theory in
  the large n limit.
\newblock {\em Communications in Mathematical Physics}, 366(1):203--268, 2019.

\bibitem{Che}
Ilya Chevyrev.
\newblock Yang-{M}ills measure on the two-dimensional torus as a random
  distribution.
\newblock {\em Comm. Math. Phys.}, 372(3):1027--1058, 2019.

\bibitem{Dah2}
Antoine Dahlqvist.
\newblock Free energies and fluctuations for the unitary {B}rownian motion.
\newblock {\em Comm. Math. Phys.}, 348(2):395--444, 2016.

\bibitem{Dah}
Antoine Dahlqvist.
\newblock Integration formulas for {B}rownian motion on classical compact {L}ie
  groups.
\newblock {\em Ann. Inst. Henri Poincar\'{e} Probab. Stat.}, 53(4):1971--1990,
  2017.

\bibitem{DL}
Antoine Dahlqvist and Thibaut Lemoine.
\newblock Large {$N$} limit of the yang-mills measure on compact surfaces {II}:
  Makeenko-migdal equations and planar master field.
\newblock 2021.

\bibitem{DN}
Antoine Dahlqvist and James~R. Norris.
\newblock Yang--mills measure and the master field on the sphere.
\newblock {\em Communications in Mathematical Physics}, 377(2):1163--1226,
  2020.

\bibitem{DK2}
Jean-Marc Daul and Vladimir~A. Kazakov.
\newblock Wilson loop for large {$N$} {Y}ang-{M}ills theory on a
  two-dimensional sphere.
\newblock {\em Phys. Lett. B}, 335(3-4):371--376, 1994.

\bibitem{MonvS}
A.~B. de~Monvel and M.~V. Shcherbina.
\newblock Free energy of the two-dimensionalu(n)-gauge field theory on the
  sphere.
\newblock {\em Theoretical and Mathematical Physics}, 115(3):670--679, 1998.

\bibitem{DiaconisEvans}
Persi Diaconis and Steven~N. Evans.
\newblock Linear functionals of eigenvalues of random matrices.
\newblock {\em Trans. Amer. Math. Soc.}, 353(7):2615--2633, 2001.

\bibitem{DK}
M.~R. Douglas and V.~A. Kazakov.
\newblock Large {$N$} {P}hase {T}ransition in continuum {QCD}${}_2$.
\newblock {\em Physics Letters {B}}, 319, 1993.

\bibitem{Dou}
Michael~R. Douglas.
\newblock Conformal field theory techniques in large {$N$} {Y}ang-{M}ills
  theory.
\newblock In {\em Quantum field theory and string theory ({C}arg\`ese, 1993)},
  volume 328 of {\em NATO Adv. Sci. Inst. Ser. B Phys.}, pages 119--135.
  Plenum, New York, 1995.

\bibitem{Dri}
Bruce~K. Driver.
\newblock Y{M{${}_2$}}: continuum expectations, lattice convergence, and
  lassos.
\newblock {\em Comm. Math. Phys.}, 123(4):575--616, 1989.

\bibitem{DGHK}
Bruce~K. Driver, Franck Gabriel, Brian~C. Hall, and Todd Kemp.
\newblock The {M}akeenko-{M}igdal equation for {Y}ang-{M}ills theory on compact
  surfaces.
\newblock {\em Comm. Math. Phys.}, 352(3):967--978, 2017.

\bibitem{DHK}
Bruce~K. Driver, Brian~C. Hall, and Todd Kemp.
\newblock Three proofs of the {M}akeenko-{M}igdal equation for {Y}ang-{M}ills
  theory on the plane.
\newblock {\em Comm. Math. Phys.}, 351(2):741--774, 2017.

\bibitem{Far}
Jacques Faraut.
\newblock {\em Analysis on {L}ie groups}, volume 110 of {\em Cambridge Studies
  in Advanced Mathematics}.
\newblock Cambridge University Press, Cambridge, 2008.

\bibitem{FM}
Benson Farb and Dan Margalit.
\newblock {\em A primer on mapping class groups}, volume~49 of {\em Princeton
  Mathematical Series}.
\newblock Princeton University Press, Princeton, NJ, 2012.

\bibitem{Gop}
Rajesh Gopakumar.
\newblock The master field in generalised {${\rm QCD}_2$}.
\newblock {\em Nuclear Phys. B}, 471(1-2):246--260, 1996.

\bibitem{GrossMat}
David~J. Gross and Andrei Matytsin.
\newblock Some properties of large-n two-dimensional yang-mills theory.
\newblock {\em Nuclear Physics B}, 437(3):541--584, 1995.

\bibitem{Gro}
Leonard Gross.
\newblock The {M}axwell equations for {Y}ang-{M}ills theory.
\newblock In {\em Mathematical quantum field theory and related topics
  ({M}ontreal, {PQ}, 1987)}, volume~9 of {\em CMS Conf. Proc.}, pages 193--203.
  Amer. Math. Soc., Providence, RI, 1988.

\bibitem{GKS}
Leonard Gross, Christopher King, and Ambar Sengupta.
\newblock Two-dimensional {Y}ang-{M}ills theory via stochastic differential
  equations.
\newblock {\em Ann. Physics}, 194(1):65--112, 1989.

\bibitem{Gur}
Robert Guralnick, Michael Larsen, and Corey Manack.
\newblock Low degree representations of simple lie groups.
\newblock {\em Proceedings of the American Mathematical Society},
  140(5):1823--1834, 2012.

\bibitem{Hal2}
Brian~C. Hall.
\newblock The large-{$N$} limit for two-dimensional {Y}ang--{M}ills theory.
\newblock {\em Comm. Math. Phys.}, 363(3):789--828, 2018.

\bibitem{KK}
V.~A. Kazakov and I.~K. Kostov.
\newblock Nonlinear strings in two-dimensional {${\rm U}(\infty )$} gauge
  theory.
\newblock {\em Nuclear Phys. B}, 176(1):199--215, 1980.

\bibitem{LZ}
Sergei~K. Lando and Alexander~K. Zvonkin.
\newblock {\em Graphs on surfaces and their applications}, volume 141 of {\em
  Encyclopaedia of Mathematical Sciences}.
\newblock Springer-Verlag, Berlin, 2004.

\bibitem{Lem}
Thibaut Lemoine.
\newblock Large {N} behaviour of the two-dimensional {Y}ang--{M}ills partition
  function.
\newblock {\em Combinatorics, Probability and Computing}, pages 1--22, 2021.

\bibitem{Lev3}
Thierry L{\'{e}}vy.
\newblock Yang-{M}ills measure on compact surfaces.
\newblock {\em Mem. Amer. Math. Soc.}, 166(790):xiv+122, 2003.

\bibitem{Lev4}
Thierry L{\'{e}}vy.
\newblock Wilson loops in the light of spin networks.
\newblock {\em J. Geom. Phys.}, 52(4):382--397, 2004.

\bibitem{Lev5}
Thierry L{\'{e}}vy.
\newblock Schur-{W}eyl duality and the heat kernel measure on the unitary
  group.
\newblock {\em Adv. Math.}, 218(2):537--575, 2008.

\bibitem{Lev2}
Thierry L{\'{e}}vy.
\newblock Two-dimensional {M}arkovian holonomy fields.
\newblock {\em Ast\'{e}risque}, (329):172, 2010.

\bibitem{Lev}
Thierry L{\'{e}}vy.
\newblock The master field on the plane.
\newblock {\em Ast\'{e}risque}, (388):ix+201, 2017.

\bibitem{Lev7}
Thierry L{\'e}vy.
\newblock {\em Two-Dimensional Quantum Yang--Mills Theory and the
  Makeenko--Migdal Equations}, pages 275--325.
\newblock Springer International Publishing, Cham, 2020.

\bibitem{LM}
Thierry L{\'{e}}vy and Myl\`ene Ma\"{\i}da.
\newblock On the {D}ouglas-{K}azakov phase transition. {W}eighted potential
  theory under constraint for probabilists.
\newblock In {\em Mod\'{e}lisation {A}l\'{e}atoire et
  {S}tatistique---{J}ourn\'{e}es {MAS} 2014}, volume~51 of {\em ESAIM Proc.
  Surveys}, pages 89--121. EDP Sci., Les Ulis, 2015.

\bibitem{LS}
Thierry L{\'{e}}vy and Ambar Sengupta.
\newblock Four chapters on low-dimensional gauge theories.
\newblock In {\em Stochastic geometric mechanics}, volume 202 of {\em Springer
  Proc. Math. Stat.}, pages 115--167. Springer, Cham, 2017.

\bibitem{Lia}
Ming Liao.
\newblock {\em L\'{e}vy processes in {L}ie groups}, volume 162 of {\em
  Cambridge Tracts in Mathematics}.
\newblock Cambridge University Press, Cambridge, 2004.

\bibitem{MM}
Yuri Makeenko and Alexander~A. Migdal.
\newblock Exact equation for the loop average in multicolor {QCD}.
\newblock {\em Physics Letters {B}}, 88:135--137, 1979.

\bibitem{Mel}
Pierre-Lo\"{\i}c M\'{e}liot.
\newblock The cut-off phenomenon for {B}rownian motions on compact symmetric
  spaces.
\newblock {\em Potential Anal.}, 40(4):427--509, 2014.

\bibitem{MenO}
P.~Menotti and E.~Onofri.
\newblock The action of su(n) lattice gauge theory in terms of the heat kernel
  on the group manifold.
\newblock {\em Nuclear Physics B}, 190(2):288--300, 1981.

\bibitem{Mig}
Alexander~A. Migdal.
\newblock Recursion equations in gauge field theories.
\newblock {\em Sov. Phys. JETP}, pages 413--418, 1975.

\bibitem{OstSeil}
K.~Osterwalder and E.~Seiler.
\newblock {\em Lattice gauge theories}, pages 26--36.
\newblock Springer Berlin Heidelberg, Berlin, Heidelberg, 1978.

\bibitem{Rus}
Boris Rusakov.
\newblock Large-{N} quantum gauge theories in two dimensions.
\newblock {\em Physics Letters B}, 303(1):95--98, 1993.

\bibitem{Sen4}
Ambar Sengupta.
\newblock The {Y}ang-{M}ills measure for {$S^2$}.
\newblock {\em J. Funct. Anal.}, 108(2):231--273, 1992.

\bibitem{Sen5}
Ambar Sengupta.
\newblock Gauge invariant functions of connections.
\newblock {\em Proc. Amer. Math. Soc.}, 121(3):897--905, 1994.

\bibitem{Sen0}
Ambar Sengupta.
\newblock Gauge theory on compact surfaces.
\newblock {\em Mem. Amer. Math. Soc.}, 126(600):viii+85, 1997.

\bibitem{Sen6}
Ambar~N. Sengupta.
\newblock Sewing {Y}ang-{M}ills measures for non-trivial bundles.
\newblock {\em Infin. Dimens. Anal. Quantum Probab. Relat. Top.},
  6(suppl.):39--52, 2003.

\bibitem{Sen}
Ambar~N. Sengupta.
\newblock Gauge theory in two dimensions: topological, geometric and
  probabilistic aspects.
\newblock In {\em Stochastic analysis in mathematical physics}, pages 109--129.
  World Sci. Publ., Hackensack, NJ, 2008.

\bibitem{Sen2}
Ambar~N. Sengupta.
\newblock Traces in two-dimensional {QCD}: the large-{$N$} limit.
\newblock In {\em Traces in number theory, geometry and quantum fields},
  Aspects Math., E38, pages 193--212. Friedr. Vieweg, Wiesbaden, 2008.

\bibitem{Sen7}
Ambar~N. Sengupta.
\newblock {\em Representing finite groups}.
\newblock Springer, New York, 2012.
\newblock A semisimple introduction.

\bibitem{Sin}
Isadore~M. Singer.
\newblock On the master field in two dimensions.
\newblock In {\em Functional analysis on the eve of the 21st century, {V}ol. 1
  ({N}ew {B}runswick, {NJ}, 1993)}, volume 131 of {\em Progr. Math.}, pages
  263--281. Birkh\"{a}user Boston, Boston, MA, 1995.

\bibitem{Sti}
John Stillwell.
\newblock {\em Classical topology and combinatorial group theory}, volume~72 of
  {\em Graduate Texts in Mathematics}.
\newblock Springer-Verlag, New York, second edition, 1993.

\bibitem{Hoo}
Gerard {'}{t}~Hooft.
\newblock A planar diagram theory for strong interactions.
\newblock {\em Nuclear Physics B}, 72(3):461 -- 473, 1974.

\bibitem{Wit}
Edward Witten.
\newblock On quantum gauge theories in two dimensions.
\newblock {\em Comm. Math. Phys.}, 141(1):153--209, 1991.

\bibitem{Xu}
Feng Xu.
\newblock A random matrix model from two-dimensional {Y}ang--{M}ills theory.
\newblock {\em Comm. Math. Phys.}, 190(2):287--307, 1997.

\bibitem{Zagier}
Don Zagier.
\newblock Values of zeta functions and their applications.
\newblock In {\em First {E}uropean {C}ongress of {M}athematics, {V}ol. {II}
  ({P}aris, 1992)}, volume 120 of {\em Progr. Math.}, pages 497--512.
  Birkh\"{a}user, Basel, 1994.

\end{thebibliography}
\end{document}